\documentclass[12pt]{article}

\usepackage[english]{babel}
\usepackage{amssymb}
\usepackage{amsfonts,amsmath,mathtools}
\usepackage{float}
\usepackage{url}
\usepackage{titlesec}
\usepackage[all]{xy}

\titleformat{\section}
{\centering\normalfont\scshape}{\thesection.}{0.5em}{}
\titleformat{\subsection}[runin]{\normalfont}{\thesubsection.}{0.5em}{\textbf}[]

\usepackage{bm}
\usepackage{enumitem}
\usepackage{faktor}
\usepackage{comment}

\usepackage[top=3cm,bottom=3cm,left=3.1cm,right=3.1cm,asymmetric]{geometry}

\usepackage{etoolbox}
\apptocmd{\thebibliography}{\setlength{\itemsep}{-3pt}}{}{}

\newcommand\blfootnote[1]{%
	\begingroup
	\renewcommand\thefootnote{}\footnote{#1}%
	\addtocounter{footnote}{-1}%
	\endgroup
}

\newtheorem{Lem}{Lemma}[section]
\newtheorem{Prop}[Lem]{Proposition}
\newtheorem{Cor}[Lem]{Corollary}
\newtheorem{Thm}[Lem]{Theorem}
\newtheorem{Def}[Lem]{Definition}
\newtheorem{Rem}[Lem]{Remark}
\newtheorem{Expl}[Lem]{Example}

\newenvironment{proof}[1][Proof]{\textrm{\em #1.} }{\hfill$\Box$\medskip\medskip}

\newcommand\Tor{\operatorname{Tor}}
\newcommand\Gin{\operatorname{Gin}}
\newcommand\Mon{{\operatorname{Mon}}}
\newcommand\supp{{\operatorname{supp}}}
\newcommand\lex{{\operatorname{lex}}}
\newcommand\pd{{\operatorname{pd}}}
\newcommand\reg{{\operatorname{reg}}}
\newcommand\Ker{\textup{Ker}}
\def\NZQ{\mathbb}
\def\NN{{\NZQ N}}
\def\ZZ{{\NZQ Z}}
\def\FF{{\NZQ F}}
\def\id{\textup{id}}
\let\emptyset\varnothing
\let\epsilon\varepsilon

\begin{document}
	
\title{\bf\normalsize\MakeUppercase{Vector-Spread Monomial Ideals and\\ Eliahou--Kervaire type resolutions}}	
\author{Antonino Ficarra}	
\newcommand{\Addresses}{{
 \footnotesize
 \textsc{Department of Mathematics and Computer Sciences, Physics and Earth Sciences, University of Messina, Viale Ferdinando Stagno d'Alcontres 31, 98166 Messina, Italy}
\begin{center}
 \textit{E-mail address}: \texttt{antficarra@unime.it}
\end{center}
}}
\date{}
\maketitle
\Addresses
\begin{abstract}
We introduce the class of vector-spread monomial ideals. This notion generalizes that of $t$-spread ideals introduced by Ene, Herzog and Qureshi. In particular, we focus on vector-spread strongly stable ideals, we compute their Koszul cycles and describe their minimal free resolution. As a consequence the graded Betti numbers and the Poincaré series are determined. Finally, we consider a generalization of algebraic shifting theory for such a class of ideals.
\blfootnote{
	\hspace{-0,3cm} \emph{Keywords:} monomial ideals, minimal graded resolution, vector-spread ideals, $t$-spread ideals, Koszul homology, Eliahou--Kervaire resolution.
	
}
\end{abstract}
\section*{Introduction}
\textit{Algebraic shifting} is one of the most powerful techniques in Combinatorial Commutative Algebra \cite[Chapter 11]{JT}. It is based on the simple idea of shifting and spreading the variables of the generators of a monomial ideal in a coherent way. The origins of this theory date back to a famous paper by Erd\"os, Ko and Rado, \textit{Intersection theorems for systems of finite sets} \cite{EKR61}, and made its way into Commutative Algebra through the work of Gil Kalai \cite{KA01}. Lately, algebraic shifting theory and monomial ideals arising from \textit{shifting operators} \cite{JT} have seen a resurgence. The \textit{$t$-spread monomial ideals} have been introduced in 2019 by Ene, Herzog and Qureshi \cite{EHQ}. The homological and combinatorial properties of these and some related classes of ideals are the subject of a large body of research \cite{ATSpread21,AC2,AFC1,ACF3,AFC2,CAC2018,AEL,CAC,MC19,CF1,RD,RDHQ,EHQ,FC1,CF2,HMRZ021b,LR,NQKR,ZZC}.\smallskip

In this paper, our aim is to investigate the more general possible class of such ideals. Let $S=K[x_1,\dots,x_n]$ be a polynomial ring, with $K$ a field. Given a vector ${\bf t}=(t_1,t_2,\dots,t_{d-1})\in\ZZ_{\ge0}^{d-1}$, $d\ge 2$, of non negative integers, we say that a monomial $u=x_{j_1}x_{j_2}\cdots x_{j_\ell}\in S$, with $j_1\le j_2\le\dots\le j_\ell$ and $\ell\le d$, is a \textit{vector-spread monomial of type ${\bf t}$} or simply a \textit{${\bf t}$-spread monomial} if
$$
j_{i+1}-j_i\ge t_i, \ \ \text{for all}\ i=1,\ldots,\ell-1.
$$
For instance, $u=x_1^3x_2x_4$ is $(0,0,1,2)$-spread, but not $(1,0,1,2)$-spread. A monomial ideal $I\subseteq S$ is a \textit{${\bf t}$-spread monomial ideal} if it is generated by ${\bf t}$-spread monomials. If $t_i=t$, for all $i$, a ${\bf t}$-spread monomial is called an \textit{ordinary} or \textit{uniform} $t$-spread monomial, see \cite{EHQ}. A ${\bf 1}=(1,1,\dots,1)$-spread monomial ideal is in particular \textit{squarefree}.\smallskip

Let ${\bf t}\in\ZZ_{\ge0}^{d-1}$, $d\ge2$. We say that a ${\bf t}$-spread monomial ideal $I\subset S=K[x_1,\dots,x_n]$ is a \textit{${\bf t}$-spread strongly stable ideal} if for any ${\bf t}$-spread monomial $u\in I$, and all $j<i$ such that $x_i$ divides $u$ and $x_j(u/x_i)$ is ${\bf t}$-spread, then $x_j(u/x_i)\in I$. For ${\bf t}={\bf 0}=(0,0,\dots,0)$ (${\bf t}={\bf 1}=(1,1,\dots,1)$) one obtains the classical notion of strongly stable (squarefree strongly stable) ideals \cite{JT}. On the other hand, strongly stable ideals have a central role in Commutative Algebra. Indeed, for a field $K$ of characteristic zero, they appear as \textit{generic initial ideals} \cite{JT}. Eliahou and Kervaire constructed their minimal free resolutions \cite{EK}. Bigatti and Hulett showed that among all homogeneous ideals with the same Hilbert function, the lexicographic ideals (which are also strongly stable) have the biggest Betti numbers \cite{AMB,HHU}.  Using shifting theory \cite{EHQ,JT,KA01}, Aramova, Herzog and Hibi extended these results to squarefree ideals \cite{AHH2,AHH}.\smallskip

In \cite{EHQ}, it was shown that ordinary $t$-spread strongly stable ideals have \textit{linear quotients} \cite{JT,ET}. For ideals with linear quotients, the graded Betti numbers can be easily computed. Moreover, if the so--called \textit{decomposition function} of $I$ is \textit{regular} \cite{ET}, then the minimal free resolution can be explicitly described. Unfortunately, even ordinary 2-spread strongly stable ideals do not have \textit{regular decomposition functions}, as noted in \cite{EHQ}. Hence, the minimal free resolution of ordinary $t$-spread strongly stable ideals could not be determined and has remained elusive ever since.

The purpose of this paper is twofold, we construct the minimal free resolution of vector-spread strongly stable ideals generalizing the \textit{Eliahou--Kervaire resolution} \cite{EK}, and extend \textit{algebraic shifting theory} to vector-spread strongly stable ideals \cite{KA01}.\smallskip

The paper is organized as follows. After some preliminary materials in Section \ref{sec1}, we introduce in Section \ref{sec2} the concept of vector-spread monomials and ideals.

In Sections \ref{sec3} and \ref{sec4} we construct the minimal free resolution of $I$ a ${\bf t}$-spread strongly stable ideal of $S$. As pointed out before, the method of linear quotients is unavailable to us, as the decomposition functions of vector-spread strongly stable ideals are, in general, non regular. Thus we use \textit{Koszul homology} as developed by Aramova and Herzog in \cite{AH}. Let $H_i({\bf x};S/I)$ the $i$th homology module of ${\bf x}=x_1,\dots,x_n$ with respect to $S/I$. Due to the isomorphism $\Tor_i^S(K,S/I)_j\cong H_i({\bf x};S/I)_j$, one can calculate the graded Betti numbers of $S/I$ as $\beta_{i,j}(S/I)=\dim_KH_i({\bf x};S/I)_j$. Thus one has to determine a basis of this $K$-vector space. To do so we compute the \textit{Koszul cycles} of $S/I$. As many examples indicate, Koszul cycles of arbitrary ${\bf t}$-spread strongly stable ideals do not have a nice expression as in the cases ${\bf t}={\bf 0},{\bf 1}$ \cite{AH,AHH2} (Remark \ref{RemarkRestiKoszulCyclesVectSpread}). To compute them we introduce the following notion (Definition \ref{def:vectSupport}). If $u=x_{j_1}x_{j_2}\cdots x_{j_\ell}$ is a ${\bf t}$-spread monomial, the \textit{${\bf t}$-spread support} of $u$ is the set
$$
\supp_{\bf t}(u):=\textstyle\bigcup\limits_{i=1}^{\ell-1}\ \big[j_i,j_i+(t_i-1)\big],
$$
where $[a,b]=\{c:a\le c\le b\}$, for $a,b\in\ZZ_{\ge1}$. Furthermore, we set $[n]=\{1,2,\dots,n\}$, for $n\ge1$. Let $G(I)$ be the unique minimal set of monomial generators of $I$. For a monomial $u\in S$, $\max(u)=\max\{j:x_j\ \textup{divides}\ u\}$. The main result of Section \ref{sec3} is\medskip

\noindent\textbf{Theorem \ref{TeorTSpreadStronlgyStableBetti}.} \textit{Let $I\subset S$ be a ${\bf t}$-spread strongly stable ideal. Then, for all $i\ge1$, the $K$-vector space $H_i({\bf x};S/I)$ has as a basis the homology classes of the Koszul cycles}
	\begin{equation*}
	e(u;\sigma)\ \ \  \textit{such that}\ \ \ u\in G(I), \ \ \ \sigma\subseteq[\max(u)-1]\setminus\supp_{\bf t}(u), \ \ \ |\sigma|=i-1.
	\end{equation*}

Firstly, we show that the elements $e(u;\sigma)$ (Definition \ref{def:KoszulCyclesVectorSpread}) are Koszul cycles. We employ an inductive argument, as a direct proof is rather tedious (Remark \ref{Rem:cyclesVectSpread}). Then, we inductively determine the basis for the Koszul homologies of $S/I$ on partial sequences of ${\bf x}$. Note that for ${\bf t}={\bf 0}$ (${\bf t}={\bf 1}$) the conditions that $\sigma$ must satisfy in Theorem \ref{TeorTSpreadStronlgyStableBetti} are the same as in \cite[Proposition 2.1]{AH} (\cite[Proposition 2.2]{AHH2}). So, we get a formula for the graded Betti numbers (Corollary \ref{Cor:BettiNumbFormulaVectSpread}) independent from the characteristic of the underlying field $K$, generalizing the known results in \cite{AH,AHH2,EK,EHQ}.\smallskip

In Section \ref{sec4}, we introduce the \textit{${\bf t}$-spread decomposition function} (Definition \ref{Def:tSpreadDecompFunct}). As a consequence, the differentials of the minimal free resolution of $S/I$ are explicitly described (Theorem \ref{Teor:diffVectSpreadI}). Examples \ref{Ex:KoszVect1}, \ref{Ex:KoszVect2}, \ref{Ex:KoszVect3} illustrate our methods.\smallskip

Finally, Section \ref{sec5} is devoted to a generalization of the algebraic shifting theory. Classically, a \textit{simplicial complex} $\Delta$ on the vertex set $[n]$ is called \textit{shifted} if for all $F\in\Delta$, all $i\in F$, $j\in[n]$, $j>i$, then $(F\setminus\{i\})\cup\{j\}\in\Delta$ \cite{JT,KA01}. Note that $\Delta$ is shifted if and only if the Stanley--Reisner ideal of $\Delta$, $I_\Delta$, is an ordinary squarefree (1-spread) strongly stable ideal \cite{JT}. The usefulness of \textit{Combinatorial shifting} comes from the fact that a simplicial complex shares the same \textit{$f$-vector} of its shifted simplicial complex \cite{JT}, and moreover the $f$-vector of the shifted complex is easier to compute.

From the algebraic point of view, \textit{Algebraic shifting} is defined as follows. Let $K$ be a field of characteristic zero. Let $\Gin(I)$ the \textit{generic initial ideal} of $I\subset S$ with respect to the reverse lexicographic order \cite{JT}, in particular $\Gin(I)$ is (${\bf 0}$-spread) strongly stable. One defines $I^s=(\Gin(I))^\sigma$, where $\sigma$ is the \textit{squarefree operator} that assign to each monomial $u=x_{i_1}x_{i_2}\cdots x_{i_d}$ the monomial $\sigma(u):=x_{i_1}x_{i_2+1}\cdots x_{i_d+(d-1)}$ and to each monomial ideal $I$, the monomial ideal $I^\sigma$ with minimal generating set $G(I^\sigma)=\{\sigma(u):u\in G(I)\}$ \cite{JT,KA01}. Then, the following properties hold
\begin{enumerate}
	\item[]
	\begin{enumerate}
	\item[$(\textup{Shift}_1)$] $I^s$ is a squarefree strongly stable monomial ideal;
	\item[$(\textup{Shift}_2)$] $I^s=I$ if $I$ is a squarefree strongly stable ideal;
	\item[$(\textup{Shift}_3)$] $I$ and $I^s$ have the same Hilbert function;
	\item[$(\textup{Shift}_4)$] If $I\subseteq J$, then $I^s\subseteq J^s$.
	\end{enumerate}
\end{enumerate}
We are mainly interested in the algebraic side of this construction. We introduce an analogous ``${\bf t}$-spread" algebraic shifting by the assignment $I^{s,{\bf t}}=(\Gin(I))^{\sigma_{{\bf 0},{\bf t}}}$, where $\sigma_{{\bf 0},{\bf t}}$ will be a suitable shifting operator. The ${\bf t}$-spread versions of $(\textup{Shift}_1)$-$(\textup{Shift}_4)$ will be established. For ${\bf t}={\bf 1}=(1,1,\dots,1)$, our construction returns the classical one. In particular, $(\Gin(I))^{\sigma_{{\bf 0},{\bf t}}}=I$, if $I$ is a ${\bf t}$-spread strongly stable ideal (Theorem \ref{GinVectSpread}).

\section{Preliminaries}\label{sec1}

\subsection{Monomial Ideals.} For the rest of the paper, $S=K[x_1,\dots,x_n]$ is the standard graded polynomial ring in $n$ indeterminates, with $K$ a field. Let $u=x_1^{a_1}x_2^{a_2}\cdots x_n^{a_n}\in S$ be a monomial, the integer $\deg(u)=a_1+a_2+\dots+a_n$ is called the \textit{degree} of $u$.

By $\Mon(S)$ we denote the set of all monomials of $S$. Whereas, $\Mon_{\ell}(S)$ denotes the set of all monomial of $S$ having degree $\ell$. For a monomial $u\in S$, the \textit{support} of $u$ is the set $\supp(u):=\{i:x_i\ \text{divides}\ u\}$. Furthermore, we set $\max(u):=\max\supp(u)$ and $\min(u):=\min\supp(u)$. For convenience we let $\min(1)=\max(1)=n$, for the monomial $1\in S$. For $I$ a monomial ideal of $S$, $G(I)$ denotes the unique minimal set of monomial generators of $I$, whereas $G(I)_\ell:=\big\{u\in G(I):\deg(u)=\ell\big\}$.

Any monomial ideal $I$ of $S$ has a unique minimal graded free resolution
$$
\FF: 0\rightarrow \bigoplus_{j\ge0}S(-j)^{\beta_{r,j}}\xrightarrow{\ d_r\ } \cdots\xrightarrow{\ d_1\ }\bigoplus_{j\ge0}S(-j)^{\beta_{0,j}}\xrightarrow{\ d_0\ }I\rightarrow0,
$$
where $S(-j)$ is the free $S$-module obtained by shifting the degrees of $S$ by $j$.
For all $i,j\ge0$, the numbers $\beta_{i,j}=\beta_{i,j}(I) = \dim_K\Tor^S_i(K,I)_j$ are called the \textit{graded Betti numbers} of $I$, and $\beta_i(I)=\sum_j\beta_{i,j}(I)$ is the $i$th \textit{total Betti number} of $I$. Furthermore, $\pd(I)=\max\{i:\beta_i(I)\ne0\}$ is the \textit{projective dimension} of $I$. Whereas, the (\textit{Castelnuovo--Mumford}) \textit{regularity} of $I$ is $\reg(I)=\max\{j:\beta_{i,i+j}(I)\ne0,\ \textup{for some}\ i\}$.

\subsection{The Koszul complex.} To compute the graded Betti numbers one can use the \textit{Koszul complex} \cite{JT}. Let ${\bf f}=f_1,\dots,f_m$ be a sequence of elements of $S$. The Koszul complex $K_{_{\text{\large$\boldsymbol{\cdot}$}}}({\bf f};S)$ attached to the sequence ${\bf f}$ is defined as follows: let $F$ be the free $S$-module with basis $e_1,\dots,e_m$. 
\begin{enumerate}
	\item[-] We let $K_i({\bf f};S)={\bigwedge}^iF$, for all $i=0,\dots,m$. A basis of the free $S$-module $K_i({\bf f};S)$ is given by the wedge products $e_\tau=e_{k_1}\wedge e_{k_2}\wedge\dots\wedge e_{k_i}$, where $\tau=\{k_1<k_2<\dots<k_i\}\subseteq[m]=\{1,\dots,m\}$, with $\deg(u_\tau)=|\tau|=i$.
	\item[-] We define the differential $\partial_i:K_i({\bf f};S)\rightarrow K_{i-1}({\bf f};S)$, $i=1,\dots,m-1$ by 
	$$
	\partial_i(e_\tau)\ =\ \sum_{\ell=1}^i(-1)^{\ell+1}f_{k_\ell}e_{\tau\setminus\{k_\ell\}}.
	$$
\end{enumerate}

We order the wedge products \textit{lexicographically}, as follows:\\ Let $\sigma=\{k_1<k_2<\dots<k_p\},\tau=\{\ell_1<\ell_2<\dots<\ell_q\}\subseteq[m]$, we define $\sigma>\tau$ if $p=q$ and for some $j\in[p]=\{1,\dots,p\}$ one has
$$
k_1=\ell_1,\ \ k_2=\ell_2,\ \ \ldots,\ \ k_{j-1}=\ell_{j-1}\ \ \text{and}\ \ k_j<\ell_j.
$$
If $\sigma>\tau$, then we set $e_\sigma>e_\tau$. For example,
$$
e_1\wedge e_2\ >\ e_1\wedge e_3\ >\ e_1\wedge e_4\ >\ e_2\wedge e_3\ >\ e_2\wedge e_4\ >\ e_3\wedge e_4.
$$
\subsection{Koszul Homology.} Let $I$ be a monomial ideal of $S$, and let $\varepsilon:S\rightarrow S/I$ be the canonical epimorphism. For all $1\le j\le n$, we let ${\bf x}_j$ be the regular sequence ${\bf x}_j=x_j,x_{j+1},\dots,x_n$. In particular, ${\bf x}_1={\bf x}=x_1,\dots,x_n$. One can define the complex $K_{_{\text{\large$\boldsymbol{\cdot}$}}}({\bf x}_j;S/I) = K_{_{\text{\large$\boldsymbol{\cdot}$}}}({\bf x}_j;S)\otimes S/I$. We set $\varepsilon(f)e_\sigma>\varepsilon(g)e_\tau$ if $e_{\sigma}>e_{\tau}$. We denote by $H_i({\bf x}_j;S/I)= H_i(K_{_{\text{\large$\boldsymbol{\cdot}$}}}({\bf x}_j;S/I))$ the $i$th homology module of ${\bf x}_j$ with respect to $S/I$, such modules are graded. If $a\in K_i({\bf x}_j;S/I)$ is a \textit{Koszul cycle}, \emph{i.e.}, $\partial_i(a)=0$, the symbol $[a]$ denotes the \textit{homology class} of $a$ in $H_i({\bf x}_j;S/I)$.\smallskip

Koszul homology allows us to calculate the graded Betti numbers. Indeed, we have the isomorphism $H_i({\bf x};S/I)_j\cong\Tor^S_i(K,S/I)_j$, see \cite[Corollary A.3.5]{JT}. Thus
\begin{equation}\label{eq:betaIS/I}
\beta_{i-1,j}(I)=\beta_{i,j}(S/I)=\dim_K H_{i}({\bf x};S/I)_j\ \ \textup{for all}\ \ i\ge1,j\ge0.
\end{equation}

We recall the following rule of multiplication: $\partial(a\wedge b)=\partial(a)\wedge b+(-1)^{\deg a}a\wedge\partial(b)$ for $a,b\in K_{_{\text{\large$\boldsymbol{\cdot}$}}}({\bf x};S/I)$, with $a$ homogeneous.

To simplify the notations, if there is no risk of confusion, we set $K_i({\bf x}_{j})=K_i({\bf x}_j;S/I)$ and $H_i({\bf x}_{j})=H_i({\bf x}_j;S/I)$, for all $i$ and all $j=1,\dots,n$. Each module $H_i({\bf x}_j)$ is a $S=K[x_1,\dots,x_n]$-module, so it is in particular a $S/({\bf x}_j)\cong K[x_1,\dots,x_{j-1}]$-module, and for $j=1$, a $K$-vector space.\smallskip

Let $1\le j\le n-1$. For all $i$, let $\alpha_i:K_{i}({\bf x}_{j+1})\rightarrow K_{i}({\bf x}_j)$ be the inclusion. We define also a homomorphism $\beta_i:K_{i}({\bf x}_j)\rightarrow K_{i-1}({\bf x}_{j+1})$ as follows:  Any element $a\in K_i({\bf x}_j)$ can be written uniquely as $a=e_j\wedge b+c$, for unique $b\in K_{i-1}({\bf x}_{j+1})$ and $c\in K_i({\bf x}_{j+1})$, we set $\beta_i(a)=b$. One immediately verifies that $\beta_i\circ\alpha_i=0$.

We can construct the following short exact sequence of complexes
\begin{equation}\label{shortexactsequenceVectKi}
0\rightarrow K_{_{\text{\large$\boldsymbol{\cdot}$}}}({\bf x}_{j+1};S/I) \xrightarrow{\ \alpha\ } K_{_{\text{\large$\boldsymbol{\cdot}$}}}({\bf x}_j;S/I) \xrightarrow{\ \beta\ } K_{_{\text{\large$\boldsymbol{\cdot}$}}}({\bf x}_{j+1};S/I)[-1]  \rightarrow 0,
\end{equation}
where $[-1]$ denotes the shifting of the homological degree by $-1$; indeed the following diagram is commutative with exact rows
\begin{displaymath}
\xymatrix{
	&   \vdots \ar[d]_{\partial_{i+2}} & \vdots \ar[d] \ar[d]_{\partial_{i+2}}  & \vdots \ar[d] \ar[d]_{\partial_{i+1}} & \\
	0 \ar[r] & K_{i+1}({\bf x}_{j+1}) \ar[d]_{\partial_{i+1}} \ar[r]^{\alpha_{i+1}} & K_{i+1}({\bf x}_j) \ar[r]^{\beta_{i+1}} \ar[d]_{\partial_{i+1}} & K_{i}({\bf x}_{j+1}) \ar[d]_{\partial_{i}} \ar[r] & 0 \\
	0 \ar[r] & K_{i}({\bf x}_{j+1}) \ar[d]_{\partial_{i}} \ar[r]^{\alpha_i} & K_{i}({\bf x}_j) \ar[r]^{\beta_i} \ar[d]_{\partial_{i}} & K_{i-1}({\bf x}_{j+1}) \ar[d]_{\partial_{i-1}} \ar[r] & 0 \\
	0 \ar[r] & K_{i-1}({\bf x}_{j+1}) \ar[d]_{\partial_{i-1}} \ar[r]^{\alpha_{i-1}} & K_{i-1}({\bf x}_j) \ar[r]^{\beta_{i-1}} \ar[d]_{\partial_{i-1}} & K_{i-2}({\bf x}_{j+1}) \ar[d]_{\partial_{i-2}} \ar[r] & 0 \\
	&   \vdots                       & \vdots                         & \vdots                        & \\
}\end{displaymath}

Therefore, we can apply the homology functor $H$ to the short exact sequence of complexes (\ref{shortexactsequenceVectKi}). By abuse of notations, we denote the induced maps $H_i(\alpha),H_i(\beta)$ again by $\alpha_i,\beta_i$, for all $i$. So, we have the long exact sequence
\begin{equation}\label{longexactsequenceHiVectorspread}
\begin{aligned}
\xymatrix@R-1.5pc{
	\cdots\ \ar[r]^{\delta_{i+1}\ \ \ \ \ \ } & H_{i+1}({\bf x}_{j+1}) \ar[r]^{\ \alpha_{i+1}} & H_{i+1}({\bf x}_j) \ar[r]^{\beta_{i+1}} & H_{i}({\bf x}_{j+1}) \ar[r]^{\ \ \ \ \ \ \delta_{i}} &\ \\
	\ar[r]^{\delta_{i}\ \ \ \ \ \ } & H_{i}({\bf x}_{j+1}) \ar[r]^{\ \ \alpha_i} & H_{i}({\bf x}_j) \ar[r]^{\beta_i\ } & H_{i-1}({\bf x}_{j+1}) \ar[r]^{\ \ \ \ \ \ \delta_{i-1}} &\ \\
	\ar[r]^{\delta_{i-1}\ \ \ \ \ \ } & H_{i-1}({\bf x}_{j+1}) \ar[r]^{\ \alpha_{i-1}} & \ \ \ \cdots\ \ \ \ar[r]^{\alpha_{0}\ } & H_{0}({\bf x}_{j}) \ar[r] & 0,
}
\end{aligned}
\end{equation}
where the maps $\delta_i$ are the \textit{connecting homomorphisms}. It can be verified that $\delta_i$ is \textit{multiplication by} $\pm x_j$, \emph{i.e.}, $\delta_i([a])=\pm x_j[a]$, see \cite[Theorem A.3.3]{JT}.\bigskip

\section{Basic concepts}\label{sec2}

From now on, ${\bf t}=(t_1,\dots,t_{d-1})$ is a vector of non negative integers and $d\ge2$.

\begin{Def}
	\rm Let $u=x_{j_1}x_{j_2}\cdots x_{j_\ell}\in S$ be a monomial of degree $\ell\le d$, with $1\le j_1\le j_2\le\dots\le j_\ell\le n$. We say that $u$ is a \textit{vector-spread monomial of type} ${\bf t}$, or simply a \emph{${\bf t}$-spread monomial}, if
	$$
	j_{i+1}-j_i\ge t_i, \ \ \text{for all}\ i=1,\ldots,\ell-1.
	$$
	
	In particular, any variable $x_j$ is ${\bf t}$-spread. We assume that $u=1$ is ${\bf t}$-spread too.
	Whereas, we say that a monomial ideal $I\subseteq S$ is a \textit{vector-spread monomial ideal of type} ${\bf t}$, or simply a \emph{${\bf t}$-spread monomial ideal}, if all monomials $u\in G(I)$ are ${\bf t}$-spread.
\end{Def}

For instance, the monomial $u=x_1^3x_2x_4x_5$ is $(0,0,1,2,1)$-spread, but it is not $(1,0,1,2,1)$-spread.
For a ${\bf t}$-spread monomial ideal $I\subset S$, it is $G(I)_k=\emptyset$ for $k>d$. Let ${\bf 0}=(0,0,\dots,0)$ be the null vector with $d-1$ components. All monomials of degree $\ell\le d$ are ${\bf 0}$-spread. If $t_i\ge1$ for all $i$, a ${\bf t}$-spread monomial is \textit{squarefree} \cite{JT}.

\begin{Def}\label{def:vectunif}
	\rm In our context, if $t_i=t\ge0$, for all $i=1,\dots,d-1$, we say that a ${\bf t}$-spread monomial (ideal) $u\in \Mon(S)$, ($I\subseteq S$), is an \textit{uniform} or \textit{ordinary ${\bf t}$-spread} monomial (ideal). Such definition agrees with that given in \cite{EHQ}. In this case, we drop the bold character ``${\bf t}$" and we simply speak of $t$-spread monomial ideals.
\end{Def}

Let $T=K[x_1,x_2,\dots,x_n,\dots]$ be the polynomial ring in infinitely many variables. We let $\Mon(T;{\bf t})$ to be the set of all ${\bf t}$-spread monomials of $T$. Analogously, $\Mon(S;{\bf t})$ denotes the set of all ${\bf t}$-spread monomials of $S$. Furthermore, for all $0\le\ell\le d$, we define the following sets
\begin{align*}
\Mon_{\ell}(T;{\bf t})&:=\big\{u\in\Mon(T;{\bf t}):\deg(u)=\ell\big\},\\
\Mon_{\ell}(S;{\bf t})&:=\big\{u\in\Mon(S;{\bf t}):\deg(u)=\ell\big\}.
\end{align*}
Note that $\Mon_{\ell}(S;{\bf t})=\Mon_{\ell}(T;{\bf t})\cap S$, $\Mon_{\ell}(S;{\bf t})=\emptyset$ for $\ell>d$, and $\Mon(S;{\bf t})$ is the disjoint union of the sets $\Mon_{\ell}(S;{\bf t})$, $\ell=0,\dots,d$.\smallskip

Sometimes, we may use the abbreviation $M_{n,\ell,{\bf t}}$ for $\Mon_{\ell}(S;{\bf t})$. For instance,
$$
M_{5,4,(1,0,2)}=\big\{x_1x_2^2x_4,x_1x_2^2x_5,x_1x_2x_3x_5,x_1x_3^2x_5,x_2x_3^2x_5\big\}.
$$

In order to compute the cardinality of the sets $M_{n,\ell,{\bf t}}$, we introduce a new \textit{shifting operator}, see \cite{AHH}. Let ${\bf 0}=(0,0,,\dots,0)$ be the null vector with $d-1$ components, we define the map $\sigma_{{\bf 0},{\bf t}}:\Mon(T;{\bf 0})\rightarrow\Mon(T;{\bf t})$, by setting $\sigma_{{\bf 0},{\bf t}}(1)=1$, $\sigma_{{\bf 0},{\bf t}}(x_i)=x_i$ and for all monomials $u=x_{j_1}x_{j_2}\cdots x_{j_{\ell}}\in\Mon(T;{\bf 0})$ with $j_1\le j_2\le\dots\le j_\ell$, $2\le\ell\le d$,
\begin{align*}
\sigma_{{\bf 0},{\bf t}}(x_{j_1}x_{j_2}\cdots x_{j_{\ell}}):=\prod_{k=1}^{\ell}x_{j_k+\sum_{s=1}^{k-1}t_s}.
\end{align*}

Whereas, $\sigma_{{\bf t},{\bf t}}:\Mon(T;{\bf t})\rightarrow\Mon(T;{\bf t})$ denotes the identity function of $\Mon(T;{\bf t})$.

\begin{Lem}
	The map $\sigma_{{\bf 0},{\bf t}}$ is a bijection.
\end{Lem}
\begin{proof}
	We define the map $\sigma_{{\bf t},{\bf 0}}:\Mon(T;{\bf t})\rightarrow\Mon(T;{\bf 0})$, by setting $\sigma_{{\bf t},{\bf 0}}(1)=1$, $\sigma_{{\bf t},{\bf 0}}(x_i)=x_i,\ \text{for all}\ i\in\NN$, and for all monomials $u=x_{j_1}x_{j_2}\cdots x_{j_{\ell}}\in\Mon(T;{\bf t})$ with $1\le j_1\le j_2\le\dots\le j_\ell$, and $2\le\ell\le d$,
	\begin{align*}
	\sigma_{{\bf t},{\bf 0}}(x_{j_1}x_{j_2}\cdots x_{j_{\ell}})&:=\textstyle\prod\limits_{k=1}^{\ell}x_{j_k-\sum_{s=1}^{k-1}t_s}.
	\end{align*}
	One immediately verifies that $\sigma_{{\bf 0},{\bf t}}\circ\sigma_{{\bf t},{\bf 0}}=\sigma_{{\bf t},{\bf t}}$ and $\sigma_{{\bf t},{\bf 0}}\circ\sigma_{{\bf 0},{\bf t}}=\sigma_{{\bf 0},{\bf 0}}$.
\end{proof}

In particular, the restriction $\sigma_{{\bf t},{\bf 0}}|_{M_{n,\ell,{\bf t}}}$ is a injective map whose image is the set $M_{n-(t_1+t_2+\ldots+t_{\ell-1}),\ell,{\bf 0}}=\Mon_{\ell}(K[x_1,\dots,x_{n-(t_1+t_2+\ldots+t_{\ell-1})}])$. Thus,

\begin{Cor}\label{corCardMnellVect}
	For all $0\le\ell\le d$,
	$$
	|M_{n,\ell,{\bf t}}|=\binom{n+(\ell-1)-\sum_{j=1}^{\ell-1}t_j}{\ell}.
	$$
\end{Cor}

Now, we introduce three fundamental classes of ${\bf t}$-spread ideals.
\begin{Def}\rm
	Let $U$ be a non empty subset of $M_{n,\ell,{\bf t}}$, $\ell\le d$. We say that
	\begin{enumerate}
		\item[-] $U$ is a ${\bf t}$\textit{-spread stable set}, if for all $u\in U$, and all $j<\max(u)$ such that $x_j(u/x_{\max(u)})$ is ${\bf t}$-spread, then $x_j(u/x_{\max(u)})\in U$;
		\item[-] $U$ is a ${\bf t}$-\textit{spread strongly stable set}, if for all $u\in U$, and all $j<i$ such that $x_i$ divides $u$ and $x_j(u/x_{i})$ is ${\bf t}$-spread, then $x_j(u/x_{i})\in U$;
		\item[-] $U$ is a ${\bf t}$-\textit{spread lexicographic set}, if for all $u\in U$, $v\in M_{n,\ell,{\bf t}}$ such that $v\ge_{\lex}u$, then $v\in U$, where $\ge_{\lex}$ is the lexicographic order with $x_1>x_2>\dots>x_n$ \cite{JT}.
	\end{enumerate}
	We assume the empty set $\emptyset$ to be a ${\bf t}$-spread stable, strongly stable and lexicographic set. Whereas, for $I$ a ${\bf t}$-spread ideal of $S$, we say that $I$ is a \textit{${\bf t}$-spread stable, strongly stable, lexicographic ideal}, if $U_\ell=I\cap M_{n,\ell,{\bf t}}$ is a ${\bf t}$-spread stable, strongly stable, lexicographic set, respectively, for all $\ell=0,\dots,d$.
\end{Def}

For ${\bf t}={\bf 0}=(0,0,\dots,0)$ we obtain the classical notions of stable, strongly stable and lexicographic sets and ideals \cite{JT}. For ${\bf t}={\bf 1}=(1,1,\dots,1)$, the \textit{squarefree} analogues \cite{AHH2}. Finally, if ${\bf t}=(t,t,\dots,t)$ we have the ordinary $t$-spread stable, strongly stable and lexicographic sets and ideals, as in \cite{EHQ}.\medskip

The following hierarchy of ${\bf t}$-spread monomial ideals of $S$ holds
\begin{center}
	${\bf t}$-spread lexicographic ideals $\implies$ ${\bf t}$-spread strongly stable ideals\\ $\implies$ ${\bf t}$-spread stable ideals.
\end{center}

The next lemma provides the existence of a \textit{standard decomposition} for all ${\bf t}$-spread monomials belonging to a ${\bf t}$-spread strongly stable ideal.
\begin{Lem}\label{LemG(I)M(I)}
	Let $I$ be a ${\bf t}$-spread strongly stable ideal of $S$, and $w\in I$ a ${\bf t}$-spread monomial. Then, there exist unique monomials $u\in G(I)$ and $v\in\Mon(S)$ such that $w=uv$ and $\max(u)\le\min(v)$.
\end{Lem}
\begin{proof}
	The statement holds when $w\in G(I)$. In such a case, $w=w\cdot1$, with $w\in G(I)$, $1\in\Mon(S)$ and $\max(w)\le n=\min(1)$. Otherwise, there exists a ${\bf t}$-spread monomial $u\in G(I)$ such that $u$ divides $w$ and $\deg(u)<\deg(w)$. We choose $u$ to be of minimal degree. Then $w=uv$, for a suitable monomial $v\in\Mon(S)$. Write $w=x_{j_1}x_{j_2}\cdots x_{j_\ell}$, then $u=x_{j_{k_1}}x_{j_{k_2}}\cdots x_{j_{k_s}}$ for $1\le k_1<k_2<\dots<k_s<\ell$. Now, $u_1=x_{j_1}x_{j_2}\cdots x_{j_s}\in I$, as $u_1$ is ${\bf t}$-spread and $I$ is ${\bf t}$-spread strongly stable. Moreover, $u_1$ divides $w$ and it is a minimal generator. Otherwise, if there exists $u_2\in G(I)$ such that $u_2$ divides $u_1$ and $\deg(u_2)<\deg(u_1)$, then $\deg(u_2)<\deg(u)$ and $u_2$ divides $w$, an absurd for the choice of $u$. Hence $w=u_1v_1$ with $u_1=x_{j_1}\cdots x_{j_s}\in G(I)$ and $v_1=x_{j_{s+1}}\cdots x_{j_\ell}\in\Mon(S)$. Clearly, the monomials $u_1$ and $v_1$ satisfying the statement are unique.
\end{proof}

As a consequence we have the following
\begin{Cor}\label{cor:G(I)VectSpreadSS}
	Let $I$ be a ${\bf t}$-spread monomial ideal of $S$. Then, the following conditions are equivalent:
	\begin{enumerate}
		\item[\textup{(i)}] $I$ is a ${\bf t}$-spread strongly stable ideal;
		\item[\textup{(ii)}] for all $u\in G(I)$, $i\in\supp(u)$, $j<i$ such that $x_j(u/x_i)$ is a ${\bf t}$-spread monomial, then $x_j(u/x_i)\in I$.
	\end{enumerate}
\end{Cor}
\begin{proof}
	(i)$\implies$(ii) is obvious. For the converse, let $w\in I$ be a ${\bf t}$-spread monomial, $i\in\supp(w)$ and $j<i$ such that $w_1=x_j(w/x_i)$ is ${\bf t}$-spread, we need to prove that $w_1\in I$. Write $w=uv$, with $u$ and $v$ as in Lemma \ref{LemG(I)M(I)}. If $i\notin\supp(u)$, then $u$ divides $w_1$, and so $w_1\in I$. Otherwise, if $i\in\supp(u)$, then $j<i\le\max(u)$ and so $j\notin\supp(v)$. Thus, $w_1=x_j(w/x_i)=x_j(u/x_i)v=u_1v$ with $u_1=x_j(u/x_i)$, and $u_1$ is ${\bf t}$-spread, as $w_1$ is. By (ii), $u_1\in I$. Hence $u_1$ divides $w_1$ and so $w_1\in I$.
\end{proof}

\section{Koszul cycles of vector-spread strongly stable ideals}\label{sec3}

The main computational tool in this paper is Theorem \ref{TeorTSpreadStronlgyStableBetti}. It allows to calculate a basis of the homology modules of the Koszul complex $K_{_{\text{\large$\boldsymbol{\cdot}$}}}({\bf x};S/I)$, where ${\bf x}=x_1,\dots,x_n$ and $I$ is a ${\bf t}$-spread strongly stable ideal of $S$.\medskip

The symbol $[n]$ denotes the set $\{1,2,\dots,n\}$, where $n\in\ZZ_{\ge1}$. If $j,k\ge1$ are integers, we set $[j,k]:=\{\ell\in\mathbb{N}:j\le\ell\le k\}$, and $[j,k]\ne\emptyset$ if and only if $j\le k$. If $j=k=0$, we set $[0,0]=[0]:=\emptyset$. For a monomial $u\in S$, $u\ne1$, we set $u'=u/x_{\max(u)}$.\medskip

The next combinatorial tool will be fundamental for our aim.
\begin{Def}\label{def:vectSupport}
	\rm Let $u=x_{j_1}x_{j_2}\cdots x_{j_{\ell}}\in\Mon(S;{\bf t})$ a ${\bf t}$-spread monomial of $S$, with $1\le j_1\le\dots\le j_{\ell}\le n$. The \textit{${\bf t}$-spread support} of $u$ is the following subset of $[n]$:
	$$
	\supp_{\bf t}(u):=\textstyle\bigcup\limits_{i=1}^{\ell-1}\ \big[j_i,j_i+(t_i-1)\big].
	$$
	Note that $\supp_{\bf 0}(u)=\emptyset$, and if $u$ is squarefree, $\supp_{\bf 1}(u)=\supp(u/x_{\max(u)})=\big\{j_1,j_2,\dots,j_{\ell-1}\big\}$, where ${\bf 1}=(1,1,\dots,1)\in\ZZ_{\ge0}^{d-1}$.
\end{Def}\smallskip

Let us explain now the \textit{combinatorial meaning} of the vector-spread support. Let $u=x_{j_1}x_{j_2}\cdots x_{j_{\ell}}\in\Mon(S;{\bf t})$, $u\ne1$. For any $k\in[\max(u)-1]\setminus\supp_{\bf t}(u)$, we define $k^{(u)}:=\min\{j\in\supp(u):j>k\}$. We note that $k^{(u)}$ always exists as $k<\max(u)$ and $\max(u)\in\supp(u)$. An easy calculation shows that $w=x_k(u/x_{k^{(u)}})$ \textit{is again a ${\bf t}$-spread monomial}, and moreover, if $I$ is a ${\bf t}$-spread strongly stable ideal and $u\in I$, then $w\in I$ also, by definition. This property will be crucial in order to construct our Koszul cycles. For instance,

\begin{Expl}
	\rm Let $u=x_1^2x_2x_4x_6x_8\in \Mon(S;(0,0,1,2,1))$, $S=K[x_1,\dots,x_8]$. We have $\supp_{(0,0,1,2,1)}(x_1x_1x_2x_4x_6x_8)=\{2,4,5,6\}$, and $[\max(u)-1]\setminus\supp_{\bf t}(u)=\{1,3,7\}$.
	For $k=1$, $k^{(u)}=2$, for $k=3$, $k^{(u)}=4$ and for $k=7$, $k^{(u)}=\max(u)=8$. Note that $x_1(u/x_2)=x_1^3x_4x_6x_8$, $x_3(u/x_4)=x_1^2x_2x_3x_6x_8$ and $x_7(u/x_8)=x_1^2x_2x_4x_6x_7$ are all $(0,0,1,2,1)$-spread monomials.
\end{Expl}\smallskip

Let $I$ be a ${\bf t}$-spread strongly stable ideal of $S$. We are going to construct suitable cycles of $K_i({\bf x};S/I)=K_i({\bf x})$.

We shall make the following conventions. For a non empty subset $A\subseteq [n]$, we set ${\bf x}_A=\prod_{i\in A}x_i$ and $e_A=\bigwedge_{i\in A}e_i$, whereas for $A=\emptyset$, ${\bf x}_{\emptyset}=1$ and $e_\tau\wedge e_{\emptyset}=e_{\emptyset}\wedge e_{\tau}=e_{\tau}$ for any non empty subset $\tau\subseteq[n]$. We take account of repetitions. For example, if $A=\{1,1,2,3\}$, then ${\bf x}_A=x_1^2x_2x_3$ and $e_A=e_1\wedge e_1\wedge e_2\wedge e_3=0$.\medskip

Let $u\in S$ be a ${\bf t}$-spread monomial and $\sigma=\{k_1<k_2<\dots<k_{i-1}\}\subseteq[\max(u)-1]$ with $|\sigma|=i-1$, $i\ge1$. For each $\ell=1,\dots,i-1$, we define
$$
k_\ell^{(u)}=j_\ell:=\min\big\{j\in\supp(u):j>k_\ell\big\}.
$$
Clearly, $j_1\le j_2\le\dots\le j_{i-1}\le \max(u)$. If $F=\{k_{s_1},\dots,k_{s_m}\}\subseteq\sigma$, we set
$$
F^{(u)}:=\{j_{s_1}\le j_{s_2}\le\dots\le j_{s_m}\}=\big\{\min\{j\in\supp(u):j>k_{s_\ell}\}:\ell=1,\dots,m\big\}.
$$

\begin{Def}\label{def:KoszulCyclesVectorSpread}
\rm	Let $u\in S$ be a ${\bf t}$-spread monomial. We set $u'=u/x_{\max(u)}$. Let $\sigma=\{k_1<k_2<\dots<k_{i-1}\}\subseteq[\max(u)-1]\setminus\supp_{\bf t}(u)$ with $|\sigma|=i-1$, $i\ge1$.
Let $I$ be any ${\bf t}$-spread strongly stable ideal of $S$ such that $u\in I$ and let $\varepsilon:S\rightarrow S/I$ be the canonical map. We define the following element of $K_i({\bf x};S/I)$:
\begin{align}\label{eq:KoszCyclesVectSpreadFormula}
e(u;\sigma)\ &:=\ \sum_{F\subseteq\sigma}(-1)^{u(\sigma;F)}\varepsilon({\bf x}_F(u'/{\bf x}_{F^{(u)}}))e_{\sigma\setminus F}\wedge e_{F^{(u)}}\wedge e_{\max(u)}\\
\nonumber&=\ \varepsilon(u')e_{\sigma}\wedge e_{\max(u)}+\sum_{\emptyset\ne F\subseteq\sigma}(-1)^{u(\sigma;F)}\varepsilon({\bf x}_F(u'/{\bf x}_{F^{(u)}}))e_{\sigma\setminus F}\wedge e_{F^{(u)}}\wedge e_{\max(u)},
\end{align}
with $u(\sigma;\emptyset):=0$ and for $F\ne\emptyset$, $F\subseteq\sigma$, $u(\sigma;F)$ is defined recursively as follows:
\begin{enumerate}
	\item[-] if $\max(\sigma)=k_{i-1}\notin F$, then $u(\sigma;F):=u(\sigma\setminus\{k_{i-1}\};F)+|F|$;
	\item[-] if $\max(\sigma)=k_{i-1}\in F$, then
\end{enumerate}
\begin{align*}
u(\sigma;F)\ := &\ u(\sigma\setminus\{k_r\in\sigma:j_r=j_{i-1}\};F\setminus\{k_r\in F:j_r=j_{i-1}\})\\
&\phantom{:}+(|\{k_r\in \sigma:j_r=j_{i-1}\}|-1)(|F|+1)+1.
\end{align*}
\end{Def}
The definition of the $u(\sigma;F)$'s will became clear in the proof of Proposition \ref{Lem:cyclesVectSpread}.\smallskip

The element $e(u;\sigma)\in K_i({\bf x})$ is well defined. Indeed, if for some $k_p<k_q$, $k_p,k_q\in\sigma$ we have $j_p=j_q=j$, then $x_j^2$ may not divide $u$, however, in such case the wedge product $e_{\sigma\setminus F}\wedge e_{F^{(u)}}\wedge e_{\max(u)}$ is zero, as $j_p,j_q\in F^{(u)}$ and $e_{j_p}\wedge e_{j_q}=e_j\wedge e_j=0$. The same reasoning applies if $j_p=\max(u)$, for some $p$. Moreover, $\varepsilon(u')e_{\sigma}\wedge e_{\max(u)}$ is the biggest summand of $e(u;\sigma)$ with respect to the order on the wedge products we have defined in Section \ref{sec1}.\medskip

As noted before, for $u\in I$ a ${\bf t}$-spread monomial and $k_\ell\in[\max(u)-1]\setminus\supp_{\bf t}(u)$, then $x_{k_\ell}(u/x_{j_\ell})\in I$, \emph{i.e.}, $\varepsilon(x_{k_\ell}(u/x_{j_\ell}))=0$, as $I$ is a ${\bf t}$-spread strongly stable ideal.

\begin{Prop}\label{Lem:cyclesVectSpread}
	Let $I\subseteq S$ be a ${\bf t}$-spread strongly stable ideal. For all $i\ge1$, the elements
	$$
	e(u;\sigma) \ \ \text{such that}\ \ u\in G(I), \ \ \sigma\subseteq[\max(u)-1]\setminus\supp_{\bf t}(u), \ \ |\sigma|=i-1,
	$$
	are cycles of $K_i({\bf x};S/I)$.
\end{Prop}
\begin{Rem}\label{Rem:cyclesVectSpread}
	\rm Proving the above proposition in the cases $i=1,2,3$ is straightforward. For the general case, $e(u;\sigma)$ is a sum of $2^{i-1}$ terms, if $|\sigma|=i-1$. A direct verification of the equation $\partial_i(e(u;\sigma))=0$ is nasty. Therefore, we employ an inductive argument. After verifying two base cases $(i=1,2)$, we assume that $e(u;\vartheta)$ is a cycle for all proper subsets $\vartheta\subset\sigma$. Depending on some cases, we suitably write $e(u;\sigma)$ in terms of the $e(u;\vartheta)$'s, equations (\ref{Lem:cyclesVectSpread:eq3'}) and (\ref{Lem:cyclesVectSpread:eq4}). It is from these equations that we obtained our coefficients $u(\sigma;F)$, $F\subseteq\sigma$, by observing that we must change sign each time we exchange two consecutive basis elements $e_k,e_\ell$ in a non zero wedge product involving them. Finally, using the rule of multiplication, $\partial(a\wedge b)=\partial(a)\wedge b+(-1)^{\deg a}a\wedge\partial(b)$ for a $a,b\in K_{_{\text{\large$\boldsymbol{\cdot}$}}}({\bf x};S/I)$, with $a$ homogeneous, we complete our proof.
\end{Rem}

We begin by giving the \textit{decomposition} for the $e(u;\sigma)$'s mentioned above.
\begin{Lem}\label{lem:decompcycles}
	Let $u\in S$ be a ${\bf t}$-spread monomial, and let $\sigma=\{k_1<\dots<k_{i-1}\}\subseteq[\max(u)-1]\setminus\supp_{\bf t}(u)$, $|\sigma|=i-1$, $i\ge2$. Then the following hold:
	\begin{enumerate}
		\item[\textup{(a)}] If $j_{i-1}=\max(u)$, setting $\tau=\sigma\setminus\{k_{i-1}\}$, then
		\begin{equation}\label{Lem:cyclesVectSpread:eq3'}
		e(u;\sigma)=-e(u;\tau)\wedge e_{k_{i-1}}.
		\end{equation}
		\item[\textup{(b)}] If $j_{i-1}\ne\max(u)$, setting $\ell=\min\{\ell\in[i-1]:j_\ell=j_{i-1}\}$, $v=x_{k_{i-1}}u/x_{j_{i-1}}$ and $\rho=\sigma\setminus\{k_\ell,k_{\ell+1},\dots,k_{i-2},k_{i-1}\}$, then
		\begin{equation}\label{Lem:cyclesVectSpread:eq4}
		e(u;\sigma)=-e(u;\tau)\wedge e_{k_{i-1}}+(-1)^{i-1-\ell}e(v;\rho)\wedge e_{k_\ell}\wedge e_{k_{\ell+1}}\wedge\dots\wedge e_{k_{i-2}}\wedge e_{j_{i-1}}.
		\end{equation}
	\end{enumerate}
\end{Lem}
\begin{proof}
	(a) Suppose that $j_{i-1}=\max(u)$. In such a case, for all $F\subseteq\sigma$ with $k_{i-1}\in F$, the corresponding term of $e(u;\sigma)$ is zero, as $e_{j_{i-1}}\wedge e_{\max(u)}=e_{\max(u)}\wedge e_{\max(u)}=0$. Moreover, for all $F\subseteq\sigma$ such that $k_{i-1}=\max(\sigma)\notin F$, that is $F\subseteq\tau$, we have $u(\sigma;F)=u(\tau;F)+|F|$, hence $(-1)^{u(\sigma;F)}(-1)^{|F|+1}=-(-1)^{u(\tau;F)+2|F|}=-(-1)^{u(\tau;F)}$. So, we obtain the desired formula (\ref{Lem:cyclesVectSpread:eq3'}),
	\begin{align*}
	e(u;\sigma)&=\phantom{-\Big(}\sum_{F\subseteq\tau}(-1)^{u(\sigma;F)}\varepsilon({\bf x}_F(u'/{\bf x}_{F^{(u)}}))e_{\sigma\setminus F}\wedge e_{F^{(u)}}\wedge e_{\max(u)}\\
	&=\phantom{-\Big(}\sum_{F\subseteq\tau}(-1)^{u(\sigma;F)}(-1)^{|F|+1}\varepsilon({\bf x}_F(u'/{\bf x}_{F^{(u)}}))e_{\tau\setminus F}\wedge e_{F^{(u)}}\wedge e_{\max(u)}\wedge e_{k_{i-1}}\\
	&=-\big(\sum_{F\subseteq\tau}(-1)^{u(\tau;F)}\varepsilon({\bf x}_F(u'/{\bf x}_{F^{(u)}}))e_{\tau\setminus F}\wedge e_{F^{(u)}}\wedge e_{\max(u)}\big)\wedge e_{k_{i-1}}\\
	&=-\phantom{\Big(}e(u;\tau)\wedge e_{k_{i-1}}.
	\end{align*}
	
	(b) Suppose $j_{i-1}\ne\max(u)$. Note that $v=x_{k_{i-1}}(u/x_{j_{i-1}})\in I$, as $u\in G(I)$, $v$ is ${\bf t}$-spread, $k_{i-1}<j_{i-1}$ and $I$ is ${\bf t}$-spread strongly stable. Moreover $\max(v)=\max(u)$, $\ell\le i-1$ and $\rho=\{k_1,k_{2},\dots,k_{\ell-1}\}\subseteq\tau\subseteq[\max(v)-1]\setminus\supp_{\bf t}(v)$. Hence, we can consider the element
	\begin{equation}
	\label{Lem:cyclesVectSpread:eq3}
	\begin{aligned}
	e(v;\rho)&:=\sum_{G\subseteq\rho}(-1)^{v(\rho;G)}\varepsilon({\bf x}_{G}(v'/{\bf x}_{G^{(v)}}))e_{\rho\setminus G}\wedge e_{G^{(v)}}\wedge e_{\max(u)}
	\end{aligned}
	\end{equation}
	where $G^{(v)}=\{\min\{s\in\supp(v):s>g\}:g\in G\}$. For $k_r\in G$, $r<\ell$, so $s_r=\min\{s\in\supp(v):s>k_r\}=j_r$, as $j_r\in\supp(u)\setminus\{j_{i-1}\}$, and $k_{i-1}>j_r$, lest $k_{i-1}\le j_r<j_{i-1}$ would imply that $j_r=j_{i-1}$, an absurd. So, for all $G\subseteq\rho$, we have $G^{(u)}=G^{(v)}$. This implies, by the definition of the coefficients, that $v(\rho;G)=u(\rho;G)$ for all $G\subseteq\rho$.
	
	Let $F\subseteq\sigma$ with $F\ne\emptyset$ such that the corresponding term of $e(u;\sigma)$ is non zero.
	
	If $k_{i-1}\notin F$, then $F\subseteq\tau$ and in (a) we have already shown that the corresponding terms of $e(u;\sigma)$ and $-e(u;\tau)\wedge e_{k_{i-1}}$ are equal.
	
	Suppose now that $k_{i-1}\in F$. Set $D=\{k_\ell,\dots,k_{i-2}\}$. We assume that $D\cap F$ is empty, otherwise $e_{F^{(u)}}=0$, as $j_r=j_{i-1}$ for some $k_r\in D\cap F$. So, we can write $F=G\cup\{k_{i-1}\}$ for a unique $G\subseteq\rho$. Thus, the relevant sum $T$ of terms of $e(u;\sigma)$ indexed by $\big\{F=G\cup\{k_{i-1}\}:G\subseteq\rho\big\}$ is, as $\max(u)=\max(v)$ and $x_{k_{i-1}}(u'/x_{j_{i-1}})=v'$,
	\small\begin{align*}
	T=&\sum_{\substack{G\cup\{k_{i-1}\}\\ G\subseteq\rho}}(-1)^{u(\sigma;G\cup\{k_{i-1}\})}\varepsilon({\bf x}_Gx_{k_{i-1}}(u'/(x_{j_{i-1}}{\bf x}_{G^{(u)}})))e_{\sigma\setminus (G\cup\{k_{i-1}\})}\wedge e_{G^{(u)}}\wedge e_{j_{i-1}}\wedge e_{\max(u)}\\
	=&\sum_{\substack{G\subseteq\rho\\ \phantom{G\cup\{k_{i-1}\}}}}(-1)^{u(\sigma;G\cup\{k_{i-1}\})}\varepsilon({\bf x}_G(v'/{\bf x}_{G^{(v)}}))e_{\rho\setminus G}\wedge e_D\wedge e_{G^{(v)}}\wedge e_{j_{i-1}}\wedge e_{\max(v)}\\
	=&\sum_{\substack{G\subseteq\rho\\ \phantom{G\cup\{k_{i-1}\}}}}(-1)^{u(\sigma;G\cup\{k_{i-1}\})+1+|D|(|G|+1)}\varepsilon({\bf x}_G(v'/{\bf x}_{G^{(v)}}))e_{\rho\setminus G}\wedge e_{G^{(v)}}\wedge e_{\max(v)}\wedge e_D\wedge e_{j_{i-1}}.
	\end{align*}\normalsize
	Now, for all $F=G\cup\{k_{i-1}\}$, with $G\subseteq\rho$, we have
	\begin{align*}
	u(\sigma;F)&= u(\sigma\setminus\{k_\ell,\dots,k_{i-1}\};F\setminus\{k_\ell,\dots,k_{i-1}\})+(|\{k_{\ell},\dots,k_{i-1}\}|-1)(|F|+1)+1\\
	&=u(\sigma\setminus\{k_\ell,\dots,k_{i-1}\};G)+(|\{k_\ell,\dots,k_{i-1}\}|-1)(|F|+1)+1\\
	&=u(\rho;G)+|D|(|F|+1)+1.
	\end{align*}
	Therefore, as $|G|+1=|F|$ and $|D|=i-1-\ell$, we have
	$$
	(-1)^{u(\sigma;F)+1+|D|(|G|+1)}=(-1)^{u(\rho;G)+2+2|D||F|+|D|}=(-1)^{i-1-\ell}(-1)^{u(\rho;G)}.
	$$
	So, we have $e(u;\sigma)=-e(u;\tau)\wedge e_{k_{i-1}}+T$, with, as $u(\rho;G)=v(\rho;G)$ for all $G\subseteq\rho$,
	\begin{align*}
	T&=\big(\sum_{G\subseteq\rho}(-1)^{i-1-\ell}(-1)^{u(\rho;G)}\varepsilon({\bf x}_G(v'/{\bf x}_{G^{(v)}}))e_{\rho\setminus G}\wedge e_{G^{(v)}}\wedge e_{\max(v)}\big)\wedge e_D\wedge e_{j_{i-1}}\\
	&=(-1)^{i-1-\ell}\big(\sum_{G\subseteq\rho}(-1)^{v(\rho;G)}\varepsilon({\bf x}_G(v'/{\bf x}_{G^{(v)}}))e_{\rho\setminus G}\wedge e_{G^{(v)}}\wedge e_{\max(v)}\big)\wedge e_D\wedge e_{j_{i-1}}\\
	&=(-1)^{i-1-\ell}e(v;\rho)\wedge e_{k_\ell}\wedge e_{k_{\ell+1}}\wedge\dots\wedge e_{k_{i-2}}\wedge e_{j_{i-1}},
	\end{align*}
	and equation (\ref{Lem:cyclesVectSpread:eq4}) holds.
\end{proof}

Finally we can prove Proposition \ref{Lem:cyclesVectSpread}.\medskip\\
\begin{proof}[Proof of Proposition \ref{Lem:cyclesVectSpread}] For $i=1$ we have $|\sigma|=0$, so $\sigma=\emptyset$, $u(\emptyset;\emptyset)=0$ by definition, and the element $e(u;\emptyset)=\varepsilon(u')e_{\max(u)}=\varepsilon(u/x_{\max(u)})e_{\max(u)}$ is clearly a cycle of $K_1({\bf x})$. Let $i\ge 2$. We proceed by induction on $i\ge 2$.

For $i=2$, $\sigma=\{k_1\}$. Let $j_1=\min\{j\in\supp(u):j>k_1\}$. We have $u(\{k_1\};\emptyset)=0$ and  $u(\{k_1\},\{k_1\})=u(\{k_1\}\setminus\{k_1\};\{k_1\}\setminus\{k_1\})+(|\{k_1\}|-1)(|\{k_1\}|+1)+1=u(\emptyset;\emptyset)+1=1$, so
$$
e(u;\sigma)=\varepsilon(u')e_{k_1}\wedge e_{\max(u)}-\varepsilon(x_{k_1}u'/x_{j_1})e_{j_1}\wedge e_{\max(u)}.
$$

If $j_1=\max(u)$, then $e(u;\sigma)=\varepsilon(u')e_{k_1}\wedge e_{\max(u)}$. In such a case,
$$
\partial_2(e(u;\sigma))=\varepsilon(x_{k_1}u')e_{\max(u)}-\varepsilon(x_{\max(u)}u')e_{k_1}=0,
$$
as $x_{k_1}u',x_{\max(u)}u'=u\in I$. Otherwise, if $j_1<\max(u)$, then
$$
\partial_2(e(u;\sigma))=\varepsilon(x_{k_1}u')e_{\max(u)}-\varepsilon(x_{\max(u)}u')e_{k_1}-\varepsilon(x_{k_1}u')e_{\max(u)}+\varepsilon(x_{k_1}(u/x_{j_1}))e_{j_1}=0,
$$
as the first and third terms cancel each other, and $x_{\max(u)}u'=u,x_{k_1}(u/x_{j_1})\in I$.\smallskip

Suppose now $i>2$. Let $\sigma=\{k_1<\dots<k_{i-2}<k_{i-1}\}$, and $\tau=\{k_1<\dots<k_{i-2}\}$. We distinguish two cases.\smallskip

(a) Suppose $j_{i-1}=\max(u)$.
By induction $e(u;\tau)$ is a cycle. By Lemma \ref{lem:decompcycles} (a),
\begin{align*}
\partial_i(e(u;\sigma))&=\partial_i(-e(u;\tau)\wedge e_{k_{i-1}})\\
&=-\partial_{i-1}(e(u;\tau))\wedge e_{k_{i-1}}-(-1)^{\deg(e(u;\tau))}x_{k_{i-1}}e(u;\tau)\\&=-(-1)^{\deg(e(u;\tau))}x_{k_{i-1}}e(u;\tau)=0.
\end{align*}
Indeed, $x_{k_{i-1}}u/x_{j_{i-1}}=x_{k_{i-1}}u/x_{\max(u)}=x_{k_{i-1}}u'\in I$. Thus, each non zero term of $x_{k_{i-1}}e(u;\tau)$ vanish, as it has coefficient $\varepsilon({\bf x}_Fx_{k_{i-1}}u'/{\bf x}_{F^{(u)}})$, and ${\bf x}_Fx_{k_{i-1}}u'/{\bf x}_{F^{(u)}}\in I$ as $I$ is a ${\bf t}$-spread strongly stable ideal. In such a case, $e(u;\sigma)$ is a cycle, as desired.\smallskip

(b) Suppose $j_{i-1}\ne\max(u)$. Set $\ell=\min\{\ell\in[i-1]:j_\ell=j_{i-1}\}$, $v=x_{k_{i-1}}u/x_{j_{i-1}}$, $\rho=\sigma\setminus\{k_\ell,\dots,k_{i-2},k_{i-1}\}$ and $D=\{k_\ell,\dots,k_{i-2}\}$. By Lemma \ref{lem:decompcycles} (b),
\begin{equation}\label{Lem:cyclesVectSpread:eq4*}
e(u;\sigma)=-e(u;\tau)\wedge e_{k_{i-1}}+(-1)^{i-1-\ell}e(v;\rho)\wedge e_{D}\wedge e_{j_{i-1}}.
\end{equation}
Let $J$ be the smallest ${\bf t}$-spread strongly stable ideal of $S$ that contains $v$. $J$ is generated only in one degree $\deg(v)=\deg(u)$, and $J\subseteq I$. By inductive hypothesis, as $|\rho|<|\sigma|$, $e(v;\rho)$ is a cycle of $K_i({\bf x};S/J)$. So, it is also a cycle of $K_i({\bf x};S/I)=K_i({\bf x})$.

By inductive hypothesis $\partial_{i-1}(e(u;\tau))=\partial_{\ell}(e(v;\rho))=0$. Since $\deg(e(u;\tau))=|\tau|$ and $\deg(e(v;\rho))=|\rho|$, by equation (\ref{Lem:cyclesVectSpread:eq4*}) we have
\begin{align*}
\partial_i(e(u;\sigma))=& -\partial_{i-1}(e(u;\tau))\wedge e_{k_{i-1}}-(-1)^{\deg(e(u;\tau))}x_{k_{i-1}}e(u;\tau)\\
+&(-1)^{i-1-\ell}[\partial_{\ell}(e(v;\rho))\wedge e_{D}\wedge e_{j_{i-1}}+(-1)^{\deg(e(v;\rho))}e(v;\rho)\wedge \partial_{i-\ell}(e_{D}\wedge e_{j_{i-1}})]\\
=&-(-1)^{|\tau|}\big(x_{k_{i-1}}e(u;\tau)-(-1)^{|\rho|-|\tau|}(-1)^{i-1-\ell}e(v;\rho)\wedge\partial_{i-\ell}(e_{D}\wedge e_{j_{i-1}})\big).
\end{align*}
We have $i-1-\ell+|\rho|-|\tau|=i-1-\ell+\ell-1-(i-2)=0$. So $(-1)^{|\rho|-|\tau|}(-1)^{i-1-\ell}=1$.

Set 
$$
f:=x_{k_{i-1}}e(u;\tau)-e(v;\rho)\wedge\partial_{i-\ell}(e_{k_\ell}\wedge e_{k_{\ell+1}}\wedge\dots\wedge e_{k_{i-2}}\wedge e_{j_{i-1}}).
$$
To show that $\partial_i(e(u;\sigma))=0$, it suffices to prove that $f$ is zero. Let $F\subseteq\tau$. The set $D\cap F=\{k_\ell,\dots,k_{i-2}\}\cap F$ can have at most one element, otherwise $e_{F^{(u)}}=0$, as shown before. Therefore, $F=G$ or $F=G\cup\{k_r\}$ for a unique $G\subseteq\rho$ and $r\in\{\ell,\dots,i-2\}$. By construction we have $G^{(u)}=G^{(v)}$ and $u(\rho;G)=v(\rho;G)$ for all $G\subseteq\rho$, as already observed in Lemma \ref{lem:decompcycles}.

Suppose $F=G$, then the corresponding term of $f$ is $a-b$, where
\begin{align*}
a&=(-1)^{u(\tau;G)}\varepsilon(x_{k_{i-1}}{\bf x}_G(u'/{\bf x}_{G^{(u)}}))e_{\tau\setminus G}\wedge e_{G^{(u)}}\wedge e_{\max(u)},
\intertext{and, as $i-1-\ell=|D|$,}
b&=(-1)^{i-1-\ell}(-1)^{v(\rho;G)}\varepsilon(x_{j_{i-1}}{\bf x}_Gx_{k_{i-1}}u'/(x_{j_{i-1}}{\bf x}_{G^{(v)}}))e_{\rho\setminus G}\wedge e_{G^{(v)}}\wedge e_{\max(v)}\wedge e_D\\
&=(-1)^{u(\rho;G)+|D|}(-1)^{|D|(|G|+1)}\varepsilon(x_{k_{i-1}}{\bf x}_G(u'/{\bf x}_{G^{(u)}}))e_{\rho\setminus G}\wedge e_D\wedge e_{G^{(u)}}\wedge e_{\max(u)}\\
&=(-1)^{u(\rho;G)+|D|(|G|+2)}\varepsilon(x_{k_{i-1}}{\bf x}_G(u'/{\bf x}_{G^{(u)}}))e_{\tau\setminus G}\wedge e_{G^{(u)}}\wedge e_{\max(u)}.
\end{align*}
We have, as $i-1-\ell=|D|$,
\begin{align*}
u(\tau;G)&=u(\tau\setminus\{k_{i-2}\};G)+|G|\\
&=u(\tau\setminus\{k_{i-3},k_{i-2}\};G)+2|G|\\
&\phantom{..}\vdots\\
&=u(\tau\setminus\{k_{\ell},\dots,k_{i-2}\};G)+(i-1-\ell)|G|\\
&=u(\rho;G)+|D|\cdot|G|.
\end{align*}
Thus, $(-1)^{u(\rho;G)+|D|(|G|+2)}=(-1)^{u(\tau;G)}$ and $a-b=0$, in this case.\smallskip

Otherwise, if $F=G\cup\{k_{r}\}$, the corresponding term of $f$ is $a-b$, where
\begin{align*}
a&=(-1)^{u(\tau;G\cup\{k_r\})}\varepsilon(x_{k_{i-1}}x_{k_r}{\bf x}_Gu'/(x_{j_{i-1}}{\bf x}_{G^{(u)}}))e_{\tau\setminus (G\cup\{k_r\})}\wedge e_{G^{(u)}}\wedge e_{j_{i-1}}\wedge e_{\max(u)}\\
&=(-1)^{u(\tau;G\cup\{k_r\})}\varepsilon(x_{k_{i-1}}x_{k_r}{\bf x}_Gu'/(x_{j_{i-1}}{\bf x}_{G^{(u)}}))e_{\tau\setminus F}\wedge e_{G^{(u)}}\wedge e_{j_{i-1}}\wedge e_{\max(u)},
\intertext{and, setting $c=v(\rho;G)+r-\ell$,}
b&=(-1)^{r-\ell}(-1)^{v(\rho;G)}\varepsilon(x_{k_r}{\bf x}_Gv'/{\bf x}_{G^{(v)}})e_{\rho\setminus G}\wedge e_{G^{(v)}}\wedge e_{\max(v)}\wedge e_{D\setminus\{k_r\}}\wedge e_{j_{i-1}}\\
&=(-1)^{c+(|D|-1)(|G|+1)+1}\varepsilon(x_{k_{i-1}}x_{k_r}{\bf x}_Gu'/({\bf x}_{G^{(u)}}x_{j_{i-1}}))e_{\rho\setminus G}\wedge e_{D\setminus\{k_r\}}\wedge e_{F^{(u)}}\wedge e_{\max(u)}\\
&=(-1)^{c+(|D|-1)(|G|+1)+1}\varepsilon(x_{k_{i-1}}x_{k_r}{\bf x}_Gu'/({\bf x}_{G^{(u)}}x_{j_{i-1}}))e_{\tau\setminus F}\wedge e_{G^{(u)}}\wedge e_{j_{i-1}}\wedge e_{\max(u)}.
\end{align*}
Now,
\small\begin{align*}
u(\tau;G\cup\{k_r\})&=u(\tau\setminus\{k_{i-2}\};G\cup\{k_r\})+(|G|+1)\\
&=u(\tau\setminus\{k_{i-3},k_{i-2}\};G\cup\{k_r\})+2(|G|+1)\\
&\phantom{..}\vdots\\
&=u(\tau\setminus\{k_{r+1},\dots,k_{i-2}\};G\cup\{k_r\})+(i-r)(|G|+1)\\
&=u(\tau\setminus\{k_{\ell},\dots,k_{i-2}\};G)+(|\{k_\ell,\dots,k_r\}|-1)(|G|+2)+1+(i-r)(|G|+1)\\
&=u(\rho;G)+(r-\ell+i-r)(|G|+1)+(r-\ell)+1\\
&=v(\rho;G)+(i-\ell)(|G|+1)+(r-\ell)+1.
\end{align*}\normalfont\normalsize
Hence, since $(-1)^{(i-\ell)(|G|+1)}=(-1)^{(|D|+1)(|G|+1)}=(-1)^{(|D|-1)(|G|+1)}$, we have
$$
(-1)^{u(\tau;G\cup\{k_r\})}=(-1)^{v(\rho;G)+(i-\ell)(|G|+1)+(r-\ell)+1}=(-1)^{c+(|D|-1)(|G|+1)+1}
$$
and $a-b=0$. Therefore, $f=0$ and $e(u;\sigma)$ is a cycle, as desired.
\end{proof}

\begin{Rem}\label{rmbeforTeorTSpread}
	\rm Let $\sigma=\{k_1<k_2<\dots<k_{i-1}\}\subseteq[\max(u)-1]\setminus\supp_{\bf t}(u)$, $i>2$. Set
	$$
	r(u;\sigma):=\sum_{\substack{F\subseteq\sigma\\ k_1\in F}}(-1)^{u(\sigma;F)}\varepsilon({\bf x}_F(u'/{\bf x}_{F^{(u)}}))e_{\sigma\setminus F}\wedge e_{F^{(u)}}\wedge e_{\max(u)}.
	$$
	We show that $e(u;\sigma)=e_{k_1}\wedge e(u;\sigma\setminus\{k_1\})+r(u;\sigma)$. For this aim, it is enough to prove that, for all $F\subseteq\sigma$ such that $k_1\notin F$, we have $(-1)^{u(\sigma;F)}=(-1)^{u(\sigma\setminus\{k_1\};F)}$. For $\sigma=\{k_1\}$ this is clear. Let $|\sigma|=i-1\ge2$. We distinguish two cases.
	
	\textsc{Case 1.} Suppose $\max(\sigma)\notin F$, then $u(\sigma;F)=u(\sigma\setminus\{k_{i-1}\};F)+|F|$. Moreover $\max(\sigma\setminus\{k_1\})=\max(\sigma)$ as $|\sigma|\ge2$. Therefore $u(\sigma\setminus\{k_1\};F)=u(\sigma\setminus\{k_1,k_{i-1}\};F)+|F|$. By induction on $|\sigma|$, $(-1)^{u(\sigma\setminus\{k_1,k_{i-1}\};F)}=(-1)^{u(\sigma\setminus\{k_{i-1}\};F)}$, and the desired conclusion follows in such a case.
	
	\textsc{Case 2.} Suppose $\max(\sigma)\in F$. Set $D=\{k_r\in F:j_r=j_{i-1}\}$ and $d=|\{k_r\in\sigma:j_r=j_{i-1}\}|$. Observe that, as $k_1\notin F$, $d=|\{k_r\in\sigma\setminus\{k_1\}:j_r=j_{i-1}\}|$. Therefore, by the definition of the coefficients, we have
	\begin{align*}
	u(\sigma;F)\ &=\ u(\sigma\setminus D;F\setminus\{k_{i-1}\})+(d-1)(|F|+1)+1,\\
	u(\sigma\setminus\{k_1\};F)\ &=\ u(\sigma\setminus(\{k_1\}\cup D);F\setminus\{k_{i-1}\})+(d-1)(|F|+1)+1.
	\end{align*}
	As $k_{i-1}\in D$ so $D\ne\emptyset$, we have $|\sigma\setminus D|<|\sigma|$. So by inductive hypothesis, $(-1)^{u(\sigma\setminus D;F\setminus\{k_{i-1}\})}=(-1)^{u(\sigma\setminus(\{k_1\}\cup D);F\setminus\{k_{i-1}\})}$, and the desired conclusion follows.
	
	Note that $e_{k_1}$ doesn't appear in $r(u;\sigma)$. Hence, we have the useful decomposition $e(u;\sigma)=e_{k_1}\wedge e(u;\sigma\setminus\{k_1\})+r(u;\sigma)$. Moreover, equations (\ref{Lem:cyclesVectSpread:eq3'}) and (\ref{Lem:cyclesVectSpread:eq4}) give us recurrence relations for our Koszul cycles.
\end{Rem}

We are in position to state and prove the main result of this Section.\smallskip
\begin{Thm}\label{TeorTSpreadStronlgyStableBetti}
	Let $I\subset S$ be a ${\bf t}$-spread strongly stable ideal. Then, for all $i\ge1$, the $K$-vector space $H_i({\bf x};S/I)$ has as a basis the homology classes of the Koszul cycles
	\begin{equation}\label{KoszulCyclesVectSpread+Rests}
	e(u;\sigma)\ \ \  \text{such that}\ \ \ u\in G(I), \ \ \ \sigma\subseteq[\max(u)-1]\setminus\supp_{\bf t}(u), \ \ \ |\sigma|=i-1.
	\end{equation}
\end{Thm}
\begin{proof}
	Let us prove the following more general statement,\medskip
	
	\textsc{Claim 1.} \textit{For all $i\ge 1$ and all $j=1,\dots,n$, a minimal generating set for $H_i({\bf x}_{j})$, as a $S/({\bf x}_j)$-module, is given by the homology classes of the Koszul cycles}
	$$
	e(u;\sigma)\ \ \textit{such that}\ \ u\in G(I), \ \ \ \sigma\subseteq\big([\max(u)-1]\setminus\supp_{\bf t}(u)\big)\cap[j,n], \ \ \ |\sigma|=i-1.
	$$
	
	We proceed by induction on $n-j\ge0$. For the base case, let $n-j=0$. We only have to consider $H_1({\bf x}_{n})$. Indeed, for $i\ge 2$, $H_i({\bf x}_n)=0$, since $K_{_{\text{\large$\boldsymbol{\cdot}$}}}(x_n;S/I)$ has length one. $H_1({\bf x}_n)$ is generated by the elements $[e(u;\emptyset)]=[\varepsilon(u/x_{\max(u)})e_{\max(u)}]$ with $u\in G(I)$ and $\max(u)=n$. Moreover $({\bf x}_n)=(x_n)$ clearly annihilates these elements, so they form a minimal generating set of $H_1({\bf x}_n)$ as a $S/({\bf x}_n)$-module.\smallskip
	
	For the inductive step, suppose $n-j>0$ and that the thesis holds for $j+1$.
	First, we consider the case $i=1$. By the sequence (\ref{longexactsequenceHiVectorspread}), we have the exact sequence
	\begin{align}
	\label{eq:KoszSpread1}H_{1}({\bf x}_{j+1}) & \xrightarrow{\ \alpha_1\ } H_{1}({\bf x}_{j}) \xrightarrow{\ \beta_1\ } H_{0}({\bf x}_{j+1})\xrightarrow{\ \delta_0\ }H_0({\bf x}_{j+1}).
	\end{align}
	By the third isomorphism theorem for commutative rings,
	\begin{align*}
	H_{0}({\bf x}_{j+1})\ &\cong\ \frac{S/I}{({\bf x}_{j+1},I)/I}
	\ \cong\ \frac{S}{(x_{j+1},x_{j+2},\dots,x_n,I)}\ \cong\ S_{\le j}/I_{\le j},
	\end{align*}
	where $S_{\le j}=K[x_1,\dots,x_j]$ and $I_{\le j}=I\cap S_{\le j}$. We observe that $I_{\le j}$ is a monomial ideal of $S_{\le j}$ with minimal generating set $G(I_{\le j})=\{u\in G(I):\max(u)\le j\}$.
	
	Let $\Ker(\delta_0)=\textup{Im}(\beta_1)$ be the kernel of the rightmost non zero map of sequence (\ref{eq:KoszSpread1}), we obtain the short exact sequence of $S/({\bf x}_{j+1})$-modules,
	\begin{equation}\label{eq:KoszSpread3}
	0\rightarrow \textup{Im}(\alpha_1) \xrightarrow{\ \alpha_1\ } H_{1}({\bf x}_{j}) \xrightarrow{\ \beta_1\ } \Ker(\delta_0)\rightarrow 0.
	\end{equation}
	By inductive hypothesis, $H_{1}({\bf x}_{j+1})$ is generated by the homology classes of the elements
	$$
	e(u;\emptyset)=\varepsilon(u/x_{\max(u)})e_{\max(u)},
	$$
	such that $\max(u)\ge j+1$ and $u\in G(I)$. These elements also generate $\textup{Im}(\alpha_1)$, as $\alpha_1$ sends these homology classes to the corresponding homology classes in $H_1({\bf x}_j)$. Whilst, $\Ker(\delta_0)$ has as a basis the elements $\varepsilon(u/x_{\max(u)})$ with $\max(u)=j$ and $u\in G(I)$. Each of these elements is pulled back in $H_1({\bf x}_j)$ to the homology class of the element $e(u;\emptyset)$, with $\max(u)=j$ and $u\in G(I)$. Moreover $({\bf x}_j)$ annihilates $H_1({\bf x}_j)$. Indeed, consider $x_\ell[e(u;\emptyset)]$, $\ell\in[j,n]$. If $\ell=\max(u)$, then $x_\ell[e(u;\emptyset)]=[0]$. If $\ell\ne\max(u)$, then $\partial_2(e_\ell\wedge e(u;\emptyset))=x_\ell e(u;\emptyset)$, so $x_\ell[e(u;\emptyset)]=[0]$. Therefore, we see that a generating set for $H_1({\bf x}_j)$ as a $S/({\bf x}_j)$-module is as given in \textsc{Claim 1}.\medskip
	
	Now, let $i>1$. By (\ref{longexactsequenceHiVectorspread}), we have the short exact sequence of $S/({\bf x}_{j+1})$-modules,
	\begin{align}\label{eq:KoszSpread4}
	0\rightarrow\textup{Im}(\alpha_i)\xrightarrow{\ \alpha_i\ }H_i({\bf x}_j)\xrightarrow{\ \beta_i\ } \Ker(\delta_{i-1})\rightarrow0.
	\end{align}
	By inductive hypothesis, a minimal generating set of the $S/({\bf x}_{j+1})$-module $H_{i-1}({\bf x}_{j+1})$ is given by the homology classes of the Koszul cycles
	$$
	e(u;\sigma)\ \ \  \text{such that}\ \ \ u\in G(I), \ \ \ \sigma\subseteq\big([\max(u)-1]\setminus\supp_{\bf t}(u)\big)\cap[j+1,n], \ \ \ |\sigma|=i-2.
	$$
	
	The map $\delta_{i-1}$ is multiplication by $\pm x_j$. We show that the minimal generating set of $\Ker(\delta_{i-1})$ is given by those elements $[e(u;\sigma)]$ of $H_{i-1}({\bf x}_{j+1})$ such that $j\notin\supp_{\bf t}(u)$.\smallskip
	
	Let $u\in G(I)$ and let $[e(u;\sigma)]$ be an element of $H_{i-1}({\bf x}_{j+1})$ as in \textsc{Claim 1}.
	
	\textsc{Case 1.} Suppose that $j\in\supp_{\bf t}(u)$, then $x_j(u/x_{\max(u)})\notin I$. So, $\pm x_je(u;\sigma)\ne0$, as $\pm x_j\varepsilon(u')e_\sigma\wedge e_{\max(u)}\ne0$. If for absurd $[e(u;\sigma)]\in\Ker(\delta_{i-1})$, then $\delta_{i-1}([e(u;\sigma)])=\pm x_j[e(u;\sigma)]=[0]$. Since $\pm x_je(u;\sigma)\ne0$, there exists $a\in K_{i}({\bf x}_{j+1})$, $a=\sum\varepsilon(u_\gamma)e_\gamma$, for some $\gamma\subseteq[j+1,n]$, $|\gamma|=i$, $u_\gamma\in S$, such that $\partial_{i}(a)=\pm x_je(u;\sigma)$. Hence,
	\begin{align*}
	x_je(u;\sigma)\ &=\ \varepsilon(x_ju/x_{\max(u)})e_\sigma\wedge e_{\max(u)}+(\textup{smaller terms})=\partial_{i}(a)\\
	&=\ \partial_{i}\big(\sum\varepsilon(u_\gamma)e_\gamma\big)=\sum_{\gamma\ :\ \gamma\setminus\{\ell\}=\sigma\cup\{\max(u)\}}\varepsilon(x_\ell u_\gamma)e_{\sigma}\wedge e_{\max(u)}+R,
	\end{align*}
	where $R$ is a sum of other terms not involving $e_{\sigma}\wedge e_{\max(u)}$. We have $x_j(u/x_{\max(u)})\notin I$, \emph{i.e.}, $\varepsilon(x_j(u/x_{\max(u)}))\ne0$. Hence, for some $e_{\gamma_0}$ occurring in $a$ and some $\ell_0\in\gamma_0$ such that $\gamma_0\setminus\{\ell_0\}=\sigma\cup\{\max(u)\}$, we must have $x_j(u/x_{\max(u)})=x_{\ell_0} u_{\gamma_0}$. We have $\ell_0\ne\max(u)$, and since $x_j(u/x_{\max(u)})=x_{\ell_0} u_{\gamma_0}$ and $j<j+1\le\max(u)$, we also have $\ell_0<\max(u)$. Now, $\max(u)\in\gamma_0$ and the term $\pm\varepsilon(x_{\max(u)} u_{\gamma_0})e_{\gamma_0\setminus\{\max(u)\}}$ appears in $R$. The inequality $\ell_0<\max(u)$ implies that $\gamma_0\setminus\{\max(u)\}>\gamma_0\setminus\{\ell_0\}=\sigma\cup\{\max(u)\}$, and since each wedge product appearing in $x_je(u;\sigma)$ is smaller than $e_\sigma\wedge  e_{\max(u)}$, we must have either $\varepsilon(x_{\max(u)} u_{\gamma_0})=0$ or there exist $\gamma_1\subseteq[j+1,n]$ and an integer $\ell_1\in\gamma_1$ such that the term $\pm\varepsilon(x_{\ell_1}u_{\gamma_1})e_{\gamma_1\setminus\{\ell_1\}}$ appears in $R$ and cancels with $\pm\varepsilon(x_{\max(u)}u_{\gamma_0})e_{\gamma_0\setminus\{\max(u)\}}$.\smallskip
	
	\textsc{Subcase 1.1.} We have $x_{\max(u)} u_{\gamma_0}\in I$. Hence,
	$$
	x_{\max(u)} u_{\gamma_0}\ =\ x_{\max(u)} x_j(u/x_{\max(u)})/x_{\ell_0}\ =\ x_j(u/x_{\ell_0})\in I.
	$$
	Therefore, $x_j(u/x_{\ell_0})\in I$, absurd. Indeed, write $u=x_{j_1}x_{j_2}\cdots x_{j_d}$, then $j\in\supp_{\bf t}(u)$ implies that $j=j_p+r$, with $0\le r\le t_p-1$, $t_p\ge1$. Moreover, $\ell_0>j$, so $\ell_0=j_q$, $q>p$. If $x_j(u/x_{\ell_0})=x_{j_1}\cdots x_{j_p}x_{j_p+r}x_{j_{p+1}}\cdots x_{j_{q-1}}x_{j_{q+1}}\cdots x_{j_d}\in I$, then for some $v\in G(I)$, $v$ divides $x_j(u/x_{\ell_0})$. As $v$ is ${\bf t}$-spread but $x_j(u/x_{\ell_0})$ is not, we have $v\ne x_j(u/x_{\ell_0})$, hence $\deg(v)<\deg(x_j(u/x_{\ell_0}))=\deg(u)$. So, $v$ must divide $x_j(u/x_{\ell_0}))/x_j=u/x_{\ell_0}$, and $u/x_{\ell_0}\in I$, absurd as $u\in G(I)$.\smallskip
	
	\textsc{Subcase 1.2.} We have $\gamma_1\setminus\{\ell_1\}=\gamma_0\setminus\{\max(u)\}=\sigma\cup\{\ell_0\}$ and $x_{\max(u)}u_{\gamma_0}=x_{\ell_1}u_{\gamma_1}$. Therefore, $\gamma_1=\sigma\cup\{\ell_0,\ell_1\}$. The term $\pm\varepsilon(x_{\ell_0}u_{\gamma_1})e_{\sigma\cup\{\ell_1\}}$ appears in $R$ and it is bigger than $\varepsilon(u')e_{\sigma}\wedge e_{\max(u)}$. So, we have two cases to consider. As before, in the first case, $\varepsilon(x_{\ell_0}u_{\gamma_1})=0$, and recalling that $x_j(u/x_{\max(u)})=x_{\ell_0} u_{\gamma_0}$, we have
	\begin{align*}
	x_{\ell_0}u_{\gamma_1}\ &=\ x_{\ell_0}(x_{\max(u)}u_{\gamma_0})/x_{\ell_1}\ =\ x_{\max(u)}(x_{\ell_0} u_{\gamma_0})/x_{\ell_1}\ =\ x_{\max(u)}x_j(u/x_{\max(u)})/x_{\ell_1}\\&=\ x_j(u/x_{\ell_1})\in I,
	\end{align*}
	with $\ell_1>j$. Arguing as in \textsc{Subcase 1.1} we obtain an absurd. Otherwise there exist $\gamma_2\subseteq[j+1,n]$ and an integer $\ell_2\in\gamma_2$ such that the term $\pm\varepsilon(x_{\ell_2}u_{\gamma_2})e_{\gamma_2\setminus\{\ell_2\}}$ appears in $R$ and cancels with $\pm\varepsilon(x_{\ell_0}u_{\gamma_1})e_{\sigma\cup\{\ell_1\}}$. We have $\gamma_2=\sigma\cup\{\ell_1,\ell_2\}$ and consider the term arising from $\gamma_2\setminus\{\ell_1\}$. We can distinguish two cases as before. After a finite number of steps $s$, we have $x_j(u/x_{\ell_s})\in I$ for some $\ell_s>j$, obtaining an absurd. Hence, $\pm x_je(u;\sigma)\notin\textup{Im}(\partial_i)$, and $\pm x_j[e(u;\sigma)]\notin\Ker(\delta_{i-1})$.\medskip
	
	\textsc{Case 2.} Suppose now $j\notin\supp_{\bf t}(u)$. By Remark \ref{rmbeforTeorTSpread},
	$$
	e(u;\sigma\cup\{j\})=e_{j}\wedge e(u;\sigma)+r(u;\sigma).
	$$
	Recalling the map $\beta_i:K_i({\bf x}_j)\rightarrow K_{i-1}({\bf x}_{j+1})$, we have that $\beta_i(e(u;\sigma\cup\{j\}))=e(u;\sigma)$. By Proposition \ref{Lem:cyclesVectSpread}, $e(u;\sigma\cup\{j\})$ is a cycle. We prove that $[e(u;\sigma\cup\{j\})]\ne[0]$ in $H_{i}({\bf x}_j)$. Suppose on the contrary that there exists $a\in K_{i+1}({\bf x}_j)$ such that $\partial_{i+1}(a)=e(u;\sigma\cup\{j\})$. Now, $a=\sum\varepsilon(u_\gamma)e_\gamma$, for some $\gamma\subseteq[j,n]$, $|\gamma|=i+1$ and $u_\gamma\in S$. So,
	\begin{align*}
	e(u;\sigma\cup\{j\})\ &=\ \varepsilon(u/x_{\max(u)})e_{\sigma\cup\{j\}}\wedge e_{\max(u)}+(\textup{smaller terms})=\partial_{i}(a)\\ &=\ \partial_{i}\big(\sum\varepsilon(u_\gamma)e_\gamma\big)=\sum_{\gamma\ :\ \gamma\setminus\{\ell\}=\sigma\cup\{j,\max(u)\}}\varepsilon(x_\ell u_\gamma)e_{\sigma\cup\{j\}}\wedge e_{\max(u)}+R,
	\end{align*}
	where $R$ is a sum of other terms not involving $e_{\sigma\cup\{j\}}\wedge e_{\max(u)}$. For some $e_{\gamma_0}$ occurring in $a$ and some $\ell_0\in\gamma_0$ such that $\gamma_0\setminus\{\ell_0\}=\sigma\cup\{j,\max(u)\}$, we must have $u/x_{\max(u)}=x_{\ell_0} u_{\gamma_0}$. We have $\ell_0\ne\max(u)$ and $j<j+1\le\max(u)$, so $\ell_0<\max(u)$. Therefore, $\gamma_0=\sigma\cup\{j,\ell_0,\max(u)\}$ and $\pm\varepsilon(x_{\max(u)} u_{\gamma_0})e_{\gamma_0\setminus\{\max(u)\}}=\pm\varepsilon(x_{\max(u)} u_{\gamma_0})e_{\sigma\cup\{j,\ell_0\}}$ appears in $R$. Now $\ell_0<\max(u)$ implies $\gamma_0\setminus\{\max(u)\}>\gamma_0\setminus\{\ell_0\}=\sigma\cup\{j,\max(u)\}$, and since each wedge product appearing in $e(u;\sigma\cup\{j\})$ is smaller than $e_{\sigma\cup\{j\}}\wedge e_{\max(u)}$, we must have either $\varepsilon(x_{\max(u)} u_{\gamma_0})=0$ or there exist $\gamma_1\subseteq[j,n]$ and $\ell_1\in\gamma_1$ such that the term $\pm\varepsilon(x_{\ell_1}u_{\gamma_1})e_{\gamma_1\setminus\{\ell_1\}}$ appears in $R$ and cancels with $\pm\varepsilon(x_{\max(u)}u_{\gamma_0})e_{\gamma_0\setminus\{\max(u)\}}$.\smallskip
	
	\textsc{Subcase 2.1.} We have $x_{\max(u)} u_{\gamma_0}\in I$. So,
	$
	x_{\max(u)} u_{\gamma_0}=x_{\max(u)}(u/x_{\max(u)})/{x_{\ell_0}}=u/x_{\ell_0}\in I,
	$
	but this is absurd, as $u$ is a minimal monomial generator of $I$.\smallskip
	
	\textsc{Subcase 2.2.} We have $\gamma_1\setminus\{\ell_1\}=\gamma_0\setminus\{\max(u)\}=\sigma\cup\{j,\ell_0\}$ and $x_{\max(u)}u_{\gamma_0}=x_{\ell_1}u_{\gamma_1}$. Therefore, $\gamma_1=\sigma\cup\{j,\ell_0,\ell_1\}$. The term $\pm\varepsilon(x_{\ell_0}u_{\gamma_1})e_{\sigma\cup\{j,\ell_1\}}$ appears in $R$ and it is bigger than $\varepsilon(u')e_{\sigma\cup\{j\}}\wedge e_{\max(u)}$. So, we have two cases to consider. As before, in the first case, $\varepsilon(x_{\ell_0}u_{\gamma_1})=0$, and recalling that $u/x_{\max(u)}=x_{\ell_0} u_{\gamma_0}$, we have
	\begin{align*}
	x_{\ell_0}u_{\gamma_1}\ &=\ x_{\ell_0}(x_{\max(u)}u_{\gamma_0})/x_{\ell_1}\ =\ x_{\max(u)}(x_{\ell_0} u_{\gamma_0})/x_{\ell_1}\ =\ x_{\max(u)}(u/x_{\max(u)})/x_{\ell_1}\\&=\ u/x_{\ell_1}\in I,
	\end{align*}
	an absurd, as $u\in G(I)$. Otherwise, we iterate the reasoning. After a finite number of steps $s$, we have $u/x_{\ell_s}\in I$, for some $\ell_s$, an absurd. Hence $e(u;\sigma\cup\{j\})\notin\textup{Im}(\partial_{i+1})$, and $[e(u;\sigma\cup\{j\})]\ne[0]$ in $H_i({\bf x}_j)$. Therefore, $\beta_i([e(u;\sigma\cup\{j\})])=[e(u;\sigma)]$, and $[e(u;\sigma)]\in\textup{Im}(\beta_i)=\Ker(\delta_{i-1})$, as desired.\medskip
	
	Finally, a basis for $\beta_i^{-1}(\Ker(\delta_{i-1}))$ is given by all the elements as in \textsc{Claim 1} such that $j\in\sigma$. By inductive hypothesis, we know a basis for $H_{i}({\bf x}_{j+1})$, and as $\alpha_i$ sends these homology classes to the corresponding homology classes of $H_{i}({\bf x}_j)$, a minimal generating set for $\textup{Im}(\alpha_i)$ is given by all the elements as in \textsc{Claim 1} such that $j\notin\sigma$.
	
	We observe that $({\bf x}_j)$ annihilates these elements. Indeed, the elements $[e(u;\sigma)]$ as in \textsc{Claim 1} minimally generate $H_i({\bf x}_j)$ as a $S/({\bf x}_{j+1})$-module. So $({\bf x}_{j+1})$ annihilates all $[e(u;\sigma)]$. It remains to prove that $x_j$ annihilates all elements $[e(u;\sigma)]$. If $j\notin\sigma$, then by definition of $e(u;\sigma)$, $e_j$ doesn't appear in the first term $\varepsilon(u/x_{\max(u)})e_\sigma\wedge e_{\max(u)}$ of $e(u;\sigma)$. We have
	$$
	\partial_{i+1}(e_j\wedge e(u;\sigma))=x_j e(u;\sigma)+e_j\wedge(-1)^{\deg(e_j)}\partial_i(e(u;\sigma))=x_j e(u;\sigma),
	$$
	so $x_j[e(u;\sigma)]=[0]$. Suppose now $j\in\sigma$. Then $\beta_i(x_j[e(u;\sigma)])=x_j[e(u;\sigma\setminus\{j\})]=[0]$, as $[e(u;\sigma\setminus\{j\})]\in\Ker(\delta_{i-1})$. Hence, $x_j[e(u;\sigma)]\in\Ker(\beta_i)=\textup{Im}(\alpha_i)$. By Remark \ref{rmbeforTeorTSpread},
	\begin{align*}
	x_j [e(u;\sigma)]\ &=\ x_j\big[e_{j}\wedge e(u;\sigma\setminus\{j\})+r(u;\sigma)\big]\ =\ x_j\big[r(u;\sigma)\big]\in\textup{Im}(\alpha_i)\subseteq H_i({\bf x}_{j+1}),
	\end{align*}
	the first summand vanishes, as $e_j\notin H_i({\bf x}_{j+1})$. If we set $a=r(u;\sigma)$, $x_j[a]$ is a cycle, and we have $\partial_{i+1}(e_j\wedge a)=-x_ja$, so $x_j\big[r(u;\sigma)\big]=[0]$ and $x_j\big[e(u;\sigma)\big]=[0]$, as desired.
	Hence, a minimal generating set for the $S/({\bf x}_j)$-module $H_i({\bf x}_j)$ is as in \textsc{Claim 1}.
	
	Finally for $j=1$, $S/({\bf x}_1)=S/(x_1,\dots,x_n)\cong K$, and \textsc{Claim 1} implies the result, as a minimal generating set of a $K$-vector space is a basis.
\end{proof}

We provide an example that demonstrate our methods.

\begin{Expl}\label{Ex:KoszVect1}
	\rm Let ${\bf t}=(1,0,2)$, and let $I=(x_{1},x_{2}x_{3}^{2},x_{2}x_{3}x_{4}x_{6},x_{2}x_4^{2}x_{6})$. We set $w_1=x_{1},w_2=x_{2}x_{3}^{2},w_3=x_{2}x_{3}x_{4}x_{6},w_4=x_{2}x_4^{2}x_{6}$. The ideal $I\subseteq S=K[x_1,\dots,x_6]$ is a ${\bf t}$-spread strongly stable ideal with minimal generating set $G(I)=\{w_1,w_2,w_3,w_4\}$. Let ${\bf x}=x_1,x_2,\dots,x_6$.  The basis for the Koszul homologies of $S/I$ are:
	\begin{enumerate}
		\item[] \begin{enumerate}
			\item[] \begin{enumerate}
				\item[\hfil $H_1({\bf x};S/I)$:] $e(w;\emptyset)\ =\ \varepsilon(w/x_{\max(w)})\ e_{\max(w)}$,\ \ \ for $w\in G(I)$;
				\item[\hfil $H_2({\bf x};S/I)$:] $w_1$ \ gives no rise to any basis element, \begin{enumerate}
					\item[$w_2$\ gives]\ $e(w_2;\{1\})\ =\ \varepsilon(x_2x_3)\ e_1\wedge e_3$,
					\item[$w_3$\ gives]\ $e(w_3;\{1\})\ =\ \varepsilon(x_2x_3x_4)\ e_1\wedge e_6$,
					\item[]\ $e(w_3;\{3\})\ =\ \varepsilon(x_2x_3x_4)\ e_3\wedge e_6$,
					\item[$w_4$\ gives]\ $e(w_4;\{1\})\ =\ \varepsilon(x_2x_4^2)\ e_1\wedge e_6$,
					\item[]\ $e(w_4;\{3\})\ =\ \varepsilon(x_2x_4^2)\ e_3\wedge e_6-\varepsilon(x_2x_3x_4)\ e_4\wedge e_6$;
				\end{enumerate}
				\item[\hfil $H_3({\bf x};S/I)$:] $w_1,w_2$ give no rise to any basis element, \begin{enumerate}
					\item[$w_3$\ gives]\ $e(w_3;\{1,3\})\ =\ \varepsilon(x_2x_3x_4)\ e_1\wedge e_3\wedge e_6$,
					\item[$w_4$\ gives]\ $e(w_4;\{1,3\})\ =\ \varepsilon(x_2x_4^2)\ e_1\wedge e_3\wedge e_6-\varepsilon(x_2x_3x_4)\ e_1\wedge e_4\wedge e_6$;
				\end{enumerate}
				\item[\hfil $H_j({\bf x};S/I)$:] $\emptyset$,\ \ \ for all $j\ge 4$. 
			\end{enumerate}
		\end{enumerate}
	\end{enumerate}
	For instance, consider $e(w_4;\{1,3\})\in K_3({\bf x};S/I)$. Then
	\begin{align*}
	\partial_3(e(w_4;\{1,3\}))\ &=\ \partial_3\big(\varepsilon(x_2x_4^2)\ e_1\wedge e_3\wedge e_6-\varepsilon(x_2x_3x_4)\ e_1\wedge e_4\wedge e_6\big)\\
	&=\ \varepsilon(x_1x_2x_4^2)\ e_3\wedge e_6-\varepsilon(x_2x_3x_4^2)\ e_1\wedge e_6+\varepsilon(x_2x_4^2x_6)\ e_1\wedge e_3\\
	&-\ \varepsilon(x_1x_2x_3x_4)\ e_4\wedge e_6+\varepsilon(x_2x_3x_4^2)\ e_1\wedge e_6-\varepsilon(x_2x_3x_4x_6)\ e_1\wedge e_4\\
	&=\ 0.
	\end{align*}
	In fact, the first, third, fourth and sixth terms vanish, as $\varepsilon(x_1x_2x_4^2)=\varepsilon(x_2x_4^2x_6)=\varepsilon(x_1x_2x_3x_4)=\varepsilon(x_2x_3x_4x_6)=0$, and the second and fifth terms are opposite.\medskip

	We illustrate how to obtain some of these elements.

	Consider $w_3=x_2x_3x_4x_6$. Then $\supp_{\bf t}(w_3)=\supp_{(1,0,2)}(x_2x_3x_4x_6)=\{2,4,5\}$ and $[\max(w_3)-1]\setminus\supp_{\bf t}(w_3)=\{1,3\}$. Let $\vartheta=\{1,3\}$, then $\vartheta^{(w_3)}=\{2,4\}$. Moreover $\max(w_3)=6\notin\vartheta^{(w_3)}$. So, we can use equation (\ref{Lem:cyclesVectSpread:eq4}) to compute the relevant Koszul cycles $e(w_3;\sigma)$. Of course, we may also use equation (\ref{eq:KoszCyclesVectSpreadFormula}). 
	
	The monomial $w_3$ gives rise to the following Koszul cycles:
	$$\begin{array}{rrrll}
	\displaystyle\sigma=\emptyset;&& e(w_3;\emptyset)\ &=\ \phantom{-}\varepsilon(w_3/x_6)e_6=\varepsilon(x_2x_3x_4)e_6, \\ \\
	\displaystyle\sigma=\{1\};&& e(w_3;\{1\})\ &=\ -e(w_3;\emptyset)\wedge e_1+e(x_1(w_3/x_2);\emptyset)\wedge e_2\\
	\displaystyle&&&=\ -\varepsilon(x_2x_3x_4)e_6\wedge e_1+\varepsilon(x_1x_3x_4)e_6\wedge e_2\\
	\displaystyle&&&=\ \phantom{-}\varepsilon(x_2x_3x_4)e_1\wedge e_6,\\ \\
	\displaystyle\sigma=\{3\};&&
	e(w_3;\{3\})\ &=\ -e(w_3;\emptyset)\wedge e_3+e(x_3(w_3/x_4);\emptyset)\wedge e_4\\
	\displaystyle&&&=\ -\varepsilon(x_2x_3x_4)e_6\wedge e_3+\varepsilon(x_2x_3^2)e_6\wedge e_4\\
	\displaystyle&&&=\ \phantom{-}\varepsilon(x_2x_3x_4)e_3\wedge e_6,\\ \\
	\displaystyle\sigma=\{1,3\};&&e(w_3;\{1,3\})\ &=\ -e(w_3;\{1\})\wedge e_3+e(x_3(w_3/x_4);\{1\})\wedge e_4\\
	\displaystyle&&&=\ -\varepsilon(x_2x_3x_4)e_1\wedge e_6\wedge e_3+e(x_2x_3^2x_6;\{1\})\wedge e_4\\
	\displaystyle&&&=\ \phantom{-}\varepsilon(x_2x_3x_4)e_1\wedge e_3\wedge e_6+\varepsilon(x_2x_3^2)e_1\wedge e_6\wedge e_4\\
	\displaystyle&&&=\ \phantom{-}\varepsilon(x_2x_3x_4)e_1\wedge e_3\wedge e_6.
	\end{array}
	$$
	
	Our computations yield the Betti table of $S/I$,
	$$\begin{matrix}
	&0&1&2&3\\\text{total:}&1&4&5&2\\\text{0:}&1&1&\text{-}&\text{-}\\\text{1:}&\text{-}&\text{-}&\text{-}&\text{-}
	\\\text{2:}&\text{-}&1&1&\text{-}\\\text{3:}&\text{-}&2&4&2\\\end{matrix}
	$$
\end{Expl}

\begin{Rem}\label{RemarkRestiKoszulCyclesVectSpread}\rm 
	The expression of our Koszul cycles is not so nice. Indeed, a basis element $e(u;\sigma)$ of $H_i({\bf x};S/I)$, $I$ a ${\bf t}$-spread strongly stable ideal, is a sum of $2^{i-1}$ wedge products! However, if ${\bf t}=(1,\dots,1,0,\dots,0)\in\ZZ_{\ge0}^{d-1}$, $d\ge2$, the element
	$$
	z(u;\sigma):=\varepsilon(u/x_{\max(u)})e_{\sigma}\wedge e_{\max(u)}.
	$$
	with $u\in G(I)$ and $\sigma\subseteq[\max(u)-1]\setminus\supp_{\bf t}(u)$ is easily seen to be a cycle. Indeed,
	\begin{align*}
	\partial_i(z(u;\sigma))\ &=\ \sum_{j=1}^{i-1}(-1)^{j+1}\varepsilon(x_{k_j}(u/x_{\max(u)}))e_{\sigma\setminus\{k_j\}}\wedge e_{\max(u)}+(-1)^{i+1}\varepsilon(u)e_\sigma\\
	&=\ 0,
	\end{align*}
	as $x_{k_j}(u/x_{\max(u)})\in I$ for all $j$ and $u\in I$, since ${\bf t}=(1,\dots,1,0,\dots,0)$. It's easy to see that the homology classes $[z(u;\sigma)]$ are non zero and $K$-independent. Hence, they form a basis for $H_{i}({\bf x})$, as the map $z:z(u;\sigma)\mapsto e(u;\sigma)$ is a bijection and the elements $e(u;\sigma)$ form a basis of $H_{i}({\bf x})$ by Theorem \ref{TeorTSpreadStronlgyStableBetti}. These Koszul cycles have been considered in the papers \cite{AH,AHH2}. But in general they are cycles only when the vector ${\bf t}$ has the form ${\bf t}=(1,\dots,1,0,\dots,0)$.
\end{Rem}
\begin{Expl}\label{Ex:KoszVect2}
	\rm Let $I=(x_1x_2\, ,x_1x_3\, ,x_1x_2^2\, ,x_1x_2x_3\, ,x_1x_2x_4\, ,x_1x_3^2\, ,x_1x_3x_4\, ,x_1x_4^2)$ be a $(1,0)$-spread strongly stable ideal of $K[x_1,x_2,x_3,x_4]$. By Remark \ref{RemarkRestiKoszulCyclesVectSpread}, since ${\bf t}=(1,0)$ and $G(I)=\big\{x_1x_2,x_1x_3,x_1x_4^2\big\}$, the relevant basis for the Koszul homologies of $S/I$ are:
	\begin{enumerate}
		\item[] \begin{enumerate}
			\item[] \begin{enumerate}
				\item[\hfil $H_1({\bf x};S/I)$:]\ \ $\varepsilon(x_1x_2/x_2)\ e_2$,\ \ \ $\varepsilon(x_1x_3/x_3)\ e_3$,\ \ \ $\varepsilon(x_1x_4^2/x_4)\ e_4$;
				\item[\hfil $H_2({\bf x};S/I)$:]\ \ $\varepsilon(x_1x_3/x_3)\ e_2\wedge e_3$,\ \ \ $\varepsilon(x_1x_4^2/x_4)\ e_2\wedge e_4$,\ \ \ $\varepsilon(x_1x_4^2/x_4)\ e_3\wedge e_4$;
				\item[\hfil $H_3({\bf x};S/I)$:]\ \ $\varepsilon(x_1x_4^2/x_4)\ e_2\wedge e_3\wedge e_4$;
				\item[\hfil $H_j({\bf x};S/I)$:]\ \ $\emptyset$,\ \ for all $j\ge 4$. 
			\end{enumerate}
		\end{enumerate}
	\end{enumerate}
	Therefore, using \emph{Macaulay2} \cite{GDS} the Betti table of $I$ is
	$$
	\begin{matrix}
	&0&1&2\\
	\text{total:}&3&3&1\\
	\text{2:}&2&1&\text{-}\\
	\text{3:}&1&2&1
	\end{matrix}
	$$
\end{Expl}

\section{The minimal free resolution of\\ vector-spread strongly stable ideals}\label{sec4}

In this Section we construct the minimal free resolution of ${\bf t}$-spread strongly stable ideals of $S$. This resolution will generalize that of Eliahou and Kervaire \cite{EK}, and also the squarefree lexsegment analogue in \cite{AHH2}. We will follow the construction given by Aramova and Herzog in \cite{AH}.\smallskip

Let $I\subset S=K[x_1,\dots,x_n]$ be a ${\bf t}$-spread strongly stable ideal. Note that, since $\Tor_i^S(K,S/I)\cong H_i({\bf x};S/I)=H_i({\bf x})$, for all $i$, the minimal free resolution of $S/I$ may be written as follows,
$$
\FF:\cdots\xrightarrow{\ d_3\ }S\otimes_K H_2({\bf x})\xrightarrow{\ d_2\ }S\otimes_K H_1({\bf x})\xrightarrow{\ d_1\ }S\otimes_K H_0({\bf x})\xrightarrow{\ d_0\ }S/I\rightarrow0.
$$
We set $F_i=S\otimes_K H_i({\bf x})$, for all $i$, and note that $F_0=S$. By Theorem \ref{TeorTSpreadStronlgyStableBetti} and also \cite{AH}, for all $i\ge1$ a basis of the graded free $S$-module $F_i$ is given by the elements,
$$
f(u;\sigma):=1\otimes(-1)^{(i-1)(i-2)/2}[e(u;\sigma)],
$$
such that $u\in G(I)$, $\sigma\subseteq[\max(u)-1]\setminus\supp_{\bf t}(u)$ and $|\sigma|=i-1$. For later use, we shall make the following convention. If $\sigma\not\subseteq[\max(u)-1]\setminus\supp_{\bf t}(u)$ we set $f(u;\sigma)=0$.\medskip

Thus, it remains to describe the differentials $d_i$, for all $i\ge0$. For this purpose, suppose the differentials $d_0,d_1,\dots,d_{i-1}$ have already been constructed such that
$$
\FF_{<i}:F_{i-1}\xrightarrow{d_{i-1}}F_{i-2}\xrightarrow{d_{i-2}}\cdots\xrightarrow{\ d_1\ }F_0\xrightarrow{\ d_0\ }S/I\rightarrow0
$$
is exact. Fix a basis element $f(u;\sigma)$ of $F_i$. Let $\mathbb{K}=K_{_{\text{\large$\boldsymbol{\cdot}$}}}({\bf x};S/I)$ be the Koszul complex attached to ${\bf x}$ with respect to $S/I$ whose $i$th module and differential are, respectively, $K_i$ and $\partial_i:K_i\rightarrow K_{i-1}$. We consider the double complex $\mathbb{K}\otimes_S\FF_{<i}$,
\begin{displaymath}
\xymatrix{
	&   \vdots \ar[d]_{\id\otimes d_{2}}  & \vdots \ar[d]_{\id\otimes d_{2}} &  & \vdots \ar[d]^{\id\otimes d_{2}}  & \\
	0 \ar[r] & K_n\otimes F_{1} \ar[d]_{\id\otimes d_{1}} \ar[r]^{\partial_n\otimes\id\ \ } & K_{n-1}\otimes F_{1} \ar[r]^{\ \ \ \ \ \partial_{n-1}\otimes\id} \ar[d]_{\id\otimes d_{1}} & \cdots \ar[r]^{\partial_1\otimes\id\ \ \ \ } & K_0\otimes F_{1} \ar[d]^{\id\otimes d_{1}} \ar[r] & 0 \\
	0 \ar[r] & K_n\otimes F_{0} \ar[d]_{\id\otimes d_{0}} \ar[r]^{\partial_n\otimes\id\ \ } & K_{n-1}\otimes F_{0} \ar[r]^{\ \ \ \ \ \partial_{n-1}\otimes\id} \ar[d]_{\id\otimes d_{0}} & \cdots \ar[r]^{\partial_1\otimes\id\ \ \ \ } & K_0\otimes F_{0} \ar[d]^{\id\otimes d_{0}} \ar[r] & 0 \\
	0 \ar[r] & K_n\otimes S/I \ar[d] \ar[r]_{\partial_n\otimes\id\ \ } & K_{n-1}\otimes S/I \ar[r]_{\ \ \ \ \ \ \partial_{n-1}\otimes\id} \ar[d] & \cdots \ar[r]_{\partial_1\otimes\id\ \ \ \ } & K_0\otimes S/I \ar[d] \ar[r] & 0 \\
	& 0  & 0 &   & 0  &
}\end{displaymath}
where ``$\id$" denotes each time a suitable identity function.\smallskip

It is known by \cite[Section 1]{AH} that to describe how the differential $d_i$ acts on $f(u;\sigma)$ it suffices to determine elements $g_j\in K_{i-j}\otimes F_j$, $j=0,\dots,i-1$, satisfying
\begin{align}
\label{eq:giEq1}(\id_{K_i}\otimes d_0)(g_0)\ &=\ (-1)^{(i-1)(i-2)/2}1\otimes e(u;\sigma),\ \ \textup{and}\\
\label{eq:giEq2}(\id_{K_{i-j-1}}\otimes d_{j+1})(g_{j+1})\ &=\ (\partial_{i-j}\otimes\id_{F_{j}})(g_j)\ \ \textup{for}\ j=0,\dots,i-2.
\end{align}

To construct such a sequence is a difficult combinatorial task. Thus we restrict ourself to the case when ${\bf t}=(1,\dots,1,0,\dots,0)$, (Remark \ref{RemarkRestiKoszulCyclesVectSpread}). In this case we can replace the cycles $e(u;\sigma)$ by the cycles $z(u;\sigma)$. In order to construct the sequence of elements satisfying equations (\ref{eq:giEq1}) and (\ref{eq:giEq2}) we need the following notion. We recall that the \textit{pure lexicographic order} is defined as follows: $x_1^{a_1}x_2^{a_2}\cdots x_n^{a_n}>_{\textup{plex}}x_1^{b_1}x_2^{b_2}\cdots x_n^{b_n}$ if and only if $a_1=b_1,\ a_2=b_2,\ \dots,\ a_{s-1}=b_{s-1}$ and $a_s>b_s$ for some $s\in\{1,\dots,n\}$.

\begin{Def}\label{Def:tSpreadDecompFunct}
	\rm Let $I\subset S$ be a ${\bf t}$-spread strongly stable ideal, ${\bf t}=(1,\dots,1,0,\dots,0)$. Let $M(I)$ be the set of all monomials belonging to $I$. We define the map $g:M(I)\rightarrow G(I)$, as follows: for $w\in M(I)$, we set $g(w):=\max_{>_{\textup{plex}}}\{u\in G(I):u\ \textup{divides}\ w\}$. The map $g$ is called the \textit{${\bf t}$-spread decomposition function} of $I$.
\end{Def}\smallskip

For $u\in G(I)$, $k\in[\max(u)-1]\setminus\supp_{\bf t}(u)$, we set
$$
u_k:=g(x_ku)\ \ \ \text{and}\ \ \ v_k:=(x_ku)/u_k.
$$
We shall need also the following notations. For a subset $\sigma$ of $[n]$ and for $k\in\sigma$ we define $\alpha(\sigma;k):=|\{s\in\sigma:s<k\}|$. For $\tau$ a subset of $\sigma$ we let $\gamma(\tau):=\sum_{k\in\sigma\setminus\tau}\alpha(\sigma;k)$. In what follows, we denote $\sigma\setminus\{k\}$ by $\sigma\setminus k$, and $\sigma\cup\{k\}$ by $\sigma\cup k$, omitting the parentheses. To further simplify the notations, we set $\id_{K_{i-j-1}}\otimes d_{j+1}=d_{j+1}$ and $\partial_{i-j}\otimes\id_{F_{j}}=\partial_{i-j}$, for all $j=0,\dots,i-2$.

The next theorem gives the desired differentials of the resolution $\FF$ of $S/I$ and generalize \cite[Theorem 2.3]{AH}. To write the elements $g_j$ more conveniently we switch the order in the tensor products, that is we think $g_{j}$ as an element of $F_j\otimes K_{i-j}$.
\begin{Thm}\label{Teor:diffVectSpreadI}
	Let $g_0=(-1)^{\frac{(i-1)(i-2)}{2}}1\otimes(u'e_{\sigma}\wedge e_{\max(u)})$, and for $j=1,\dots,i-1$ let
	\begin{align*}
	g_j=(-1)^{i-j}\sum_{\substack{\tau\subset\sigma\\ |\tau|=j-1}}(-1)^{\gamma(\tau)}f(u;\tau)\otimes e_{\sigma\setminus\tau}
	+\sum_{\substack{\tau\subset\sigma\\ |\tau|=j}}(-1)^{\gamma(\tau)}s_\tau\otimes e_{\sigma\setminus\tau}\wedge e_{\max(u)},
	\end{align*}
	where
	$$
	s_{\tau}=\sum_{k\in\tau}(-1)^{\alpha(\tau;k)}\frac{v_k}{x_{\max(u)}}f(u_k;\tau\setminus k).
	$$
	Then the elements $g_0,g_1,\dots,g_{i-1}$ satisfy equations \textup{(\ref{eq:giEq1})} and \textup{(\ref{eq:giEq2})}. Moreover, the $i$th differential of the minimal free resolution of $S/I$ acting on $f(u;\sigma)$ is given by
	\begin{align}
	\nonumber d_i(f(u;\sigma))\ &=\ \partial_1(g_{i-1})\\
	\label{eq:diffVectSpreadRisMin}&=\ \sum_{k\in\sigma}(-1)^{\alpha(\sigma;k)}\big(-x_kf(u;\sigma\setminus k)+v_kf(u_k;\sigma\setminus k)\big).
	\end{align}
\end{Thm}
\begin{proof}
	Note that in the definition of $s_\tau$, $x_{\max(u)}$ always divide $v_k$. Indeed, if we let $k^{(u)}=\min\{j\in\supp(u):j>k\}$, then $w=x_k(u/x_{k^{(u)}})\in I$ is again a ${\bf t}$-spread monomial. By Lemma \ref{LemG(I)M(I)}, $w=w_1w_2$ with $w_1\in G(I)$ and $\max(w_1)\le\min(w_2)$. Consequently $\{y\in G(I):y\ \textup{divides}\ x_ku\}$ is non empty and $u_k$ exists. Proceeding as in the proof of Lemma \ref{LemG(I)M(I)} we see that $\max(u_k)\le\min(v_k)$. Finally, $v_k\ne 1$ otherwise $u_k=x_ku\in G(I)$, which is absurd. Hence $x_{\max(u)}$ divides $v_k$ as wanted.\smallskip
	
	We proceed by induction on $i$. The case $i=1$ is trivial. By induction, we can assume that the last formula for the differential $d_\ell$ holds for $\ell<i$. We need to verify the equations $\partial_{i-j}(g_j)=d_{j+1}(g_{j+1})$. For $j=0$ this is trivial. Let $j>0$.
	
	Firstly, we calculate $\partial_{i-j}(g_j)$. Since $|\sigma\setminus\tau|=i-j-1$, we have
	\begin{align*}
	&\partial_{i-j}(g_j)=(-1)^{i-j}\sum_{\substack{\tau\subset\sigma\\ |\tau|=j-1}} (-1)^{\gamma(\tau)}f(u;\tau)\otimes\big(\sum_{k\in\sigma\setminus\tau}(-1)^{\alpha(\sigma\setminus\tau;k)}x_k e_{\sigma\setminus(\tau\cup k)}\big)\\
	&+\sum_{\substack{\tau\subset\sigma\\ |\tau|=j}}(-1)^{\gamma(\tau)}s_\tau\otimes\big(\sum_{k\in\sigma\setminus\tau}(-1)^{\alpha(\sigma\setminus\tau;k)}x_k e_{\sigma\setminus(\tau\cup k)}\wedge e_{\max(u)}+(-1)^{i-j-1}x_{\max(u)}e_{\sigma\setminus\tau}\big).
	\end{align*}
	
	We suitably rewrite both sums.
	
	For the first sum, note that for $\tau\subseteq\sigma$, $|\tau|=j-1$ and $k\in\sigma\setminus\tau$, then setting $\rho=\tau\cup k$ we have that $|\rho|=j$, $\gamma(\tau)=\gamma(\rho\setminus k)=\sum_{s\in\sigma\setminus(\rho\cup k)}\alpha(\sigma;s)=\gamma(\rho)+\alpha(\sigma;k)$, $\alpha(\sigma;k)=\alpha(\sigma\setminus\tau;k)+\alpha(\tau;k)$ and also $\alpha(\tau;k)=\alpha(\rho;k)$ for it is $k\notin\tau$.
	Hence,
	$$
	(-1)^{\gamma(\tau)}(-1)^{\alpha(\sigma\setminus\tau;k)}=(-1)^{\gamma(\rho)+\alpha(\sigma\setminus\tau;k)+\alpha(\tau;k)}(-1)^{\alpha(\sigma\setminus\tau;k)}=(-1)^{\gamma(\rho)}(-1)^{\alpha(\rho;k)}.
	$$
	As $\tau\subseteq\sigma$, $|\tau|=j-1$ and $k\in\sigma\setminus\tau$ are arbitrary, $\rho=\tau\cup k\subseteq\sigma$ with $|\rho|=j$ is arbitrary too, thus the first sum of $\partial_{i-j}(g_j)$ can be rewritten as follows,
	\begin{align}
	\label{eq:MinRisVectSpread1}A&=(-1)^{i-j-1}\sum_{\substack{\rho\subseteq\sigma\\ |\rho|=j}}(-1)^{\gamma(\rho)}\big(\sum_{k\in\rho}(-1)^{\alpha(\rho;k)+1}x_kf(u;\rho\setminus k)\big)\otimes e_{\sigma\setminus\rho}.\\
	\intertext{Analogously, the second sum can be written as $B+C$, where}
	\label{eq:MinRisVectSpread2}B&=\sum_{\substack{\vartheta\subseteq\sigma\\ |\vartheta|=j+1}}(-1)^{\gamma(\vartheta)}\big(\sum_{k\in\vartheta}(-1)^{\alpha(\vartheta;k)}x_ks_{\vartheta\setminus k}\big)\otimes e_{\sigma\setminus\vartheta}\wedge e_{\max(u)},\\
	\label{eq:MinRisVectSpread3}C&=\phantom{\sum_{\substack{\rho\subseteq\sigma\\ |\rho|=j+1}}}(-1)^{i-j-1}\sum_{\substack{\rho\subseteq\sigma\\ |\rho|=j}}(-1)^{\gamma(\rho)}s_{\rho}\otimes x_{\max(u)}e_{\sigma\setminus\rho}.
	\end{align}
	
	Taking into account equations (\ref{eq:MinRisVectSpread1}), (\ref{eq:MinRisVectSpread3}), the inductive hypothesis and the definition of $s_\rho$ we have that
	\begin{align*}
	A+C&=(-1)^{i-(j+1)}\sum_{\substack{\rho\subset\sigma\\ |\rho|=j}}(-1)^{\gamma(\rho)}\big(\sum_{k\in\rho}(-1)^{\alpha(\rho;k)+1}x_kf(u;\rho\setminus k)+x_{\max(u)}s_\rho\big)\otimes e_{\sigma\setminus\rho}\\
	&=(-1)^{i-(j+1)}\sum_{\substack{\rho\subset\sigma\\ |\rho|=j}}(-1)^{\gamma(\rho)}d_{j+1}(f(u;\rho))\otimes e_{\sigma\setminus\rho}.
	\end{align*}
	Thus, by the structure of $g_{j+1}$, to complete our proof we need to prove that
	$$
	B=\sum_{\substack{\vartheta\subseteq\sigma\\ |\vartheta|=j+1}}(-1)^{\gamma(\vartheta)}d_{j+1}(s_{\vartheta})\otimes e_{\sigma\setminus\vartheta}\wedge e_{\max(u)}.
	$$
	That is, we have to prove
	\begin{align}
	\nonumber d_{j+1}(s_{\vartheta})&=\sum_{k\in\vartheta}(-1)^{\alpha(\vartheta;k)}x_ks_{\vartheta\setminus k}\\
	\label{eq:doublesumMinRisVectSpread}&=\sum_{k\in\vartheta}\sum_{r\in\vartheta\setminus k}(-1)^{\alpha(\vartheta;k)+\alpha(\vartheta\setminus k;r)}x_k\frac{v_r}{x_{\max(u)}}f(u_r;\vartheta\setminus\{k,r\}),
	\end{align}
	for all $\vartheta\subseteq\sigma$, $|\vartheta|=j+1$.\smallskip
	
	Since $j<i-1$, then $j+1<i$, and by inductive hypothesis,
	\begin{align*}
	d_{j+1}(s_{\vartheta})=&\sum_{r\in\vartheta}(-1)^{\alpha(\vartheta;r)}\frac{v_r}{x_{\max(u)}}\big(\sum_{k\in\vartheta\setminus r}(-1)^{\alpha(\vartheta\setminus r;k)}\times\\
	&\big(-x_kf(u_r;\vartheta\setminus\{k,r\})+\frac{x_ku_r}{g(x_ku_r)}f(g(x_ku_r);\vartheta\setminus\{k,r\})\big).
	\end{align*}
	For $u\in G(I)$ and $\tau\subseteq[n]$, we define
	$$
	\Gamma(u;\tau):=\big\{r\in\tau:\tau\setminus r\subseteq[\max(u_r)-1]\setminus\supp_{\bf t}(u_r)\big\},
	$$
	where $u_r=g(x_ru)$. Note that for $r\in\tau\setminus\Gamma(u;\tau)$, $f(u_r;\tau\setminus r)=0$.\medskip
	
	Now, for the first sum of terms of $d_{j+1}(s_\vartheta)$, note that $\alpha(\vartheta;r)=\alpha(\vartheta\setminus k;r)+\alpha(k;r)$, $\alpha(\vartheta;k)=\alpha(\vartheta\setminus r;k)+\alpha(r;k)$ and $\alpha(k;r)-\alpha(r;k)=1$ if $k<r$ or $-1$ if $k>r$. Thus,
	\begin{align*}
	(-1)^{\alpha(\vartheta;r)+\alpha(\vartheta\setminus r;k)+1}&=(-1)^{\alpha(\vartheta\setminus k;r)+\alpha(k;r)+\alpha(\vartheta;k)-\alpha(r;k)+1}=(-1)^{\alpha(\vartheta;k)+\alpha(\vartheta\setminus k;r)}.
	\end{align*}
	Taking into account this calculation and exchanging the indices $k$ with $r$ in the second sum of terms of $d_{j+1}(s_{\vartheta})$, we can write $d_{j+1}(s_\vartheta)$ as $B_1+B_2$, where
	\begin{align*}
	B_1&=\sum_{r\in\Gamma(u;\vartheta)}\sum_{k\in\vartheta\setminus r}(-1)^{\alpha(\vartheta;k)+\alpha(\vartheta\setminus k;r)}x_k\frac{v_r}{x_{\max(u)}}f(u_r;\vartheta\setminus\{k,r\}),\\
	B_2&=\sum_{\substack{k\in\Gamma(u;\vartheta)\\ r\in\Gamma(u_k;\vartheta\setminus k)}}(-1)^{\alpha(\vartheta;k)+\alpha(\vartheta\setminus k;r)}\frac{x_ru_kv_k}{g(x_ru_k)x_{\max(u)}}f(g(x_ru_k);\vartheta\setminus\{k,r\}).
	\end{align*}
	
	In all terms of the right--hand side in equation (\ref{eq:doublesumMinRisVectSpread}), for $k,r\in\vartheta$, $k\ne r$, we have either $r\in\Gamma(u;\vartheta)$ or $r\notin\Gamma(u;\vartheta)$ and $r\in\Gamma(u;\vartheta\setminus k)$. Let $B_3$ be the sum of terms such that $r\in\Gamma(u;\vartheta)$, and let $B_4$ be the sum of terms such that $r\notin\Gamma(u;\vartheta)$ and $r\in\Gamma(u;\vartheta\setminus k)$. To finish the proof, it is enough to show that $B_1=B_3$ and $B_2=B_4$.\medskip
	
	It is clear that $B_1=B_3$.\smallskip
	
	Let us see that $B_2=B_4$. The hypotheses $r\notin\Gamma(u;\vartheta)$ and $r\in\Gamma(u;\vartheta\setminus k)$ imply that $k\notin[\max(u_r)-1]\setminus\supp_{\bf t}(u_r)$, where $u_rv_r=x_ru$ and $\max(u_r)\le\min(v_r)$. But $k\in\vartheta\subseteq\sigma\subseteq[\max(u)-1]\setminus\supp_{\bf t}(u)$. Thus, either $k\in[\max(u_r),\max(u)-1]$ or $k\in\supp_{\bf t}(u_r)\setminus\supp_{\bf t}(u)$. We show in both cases that $g(x_ru_k)=u_r$.\medskip
	
	If $k\in[\max(u_r),\max(u)-1]$, then $k\ge\max(u_r)\ge r$, so $k>r$ since $k\ne r$. This implies that $r<k\le\max(u_k)$ too. Hence, $u_r$ divides $x_ru_k$. Finally, $g(x_ru_k)=u_r$.\smallskip
	
	If $k\in\supp_{\bf t}(u_r)\setminus\supp_{\bf t}(u)$, then $k>r$, and so $r<\max(u_k)$. Since $k\in\supp_{\bf t}(u_r)$ we have that $k<\max(u_r)$. Let us see that $\max(u_r)\le\max(u_k)$. Suppose on the contrary that $\max(u_r)>\max(u_k)$. If $u=x_{j_1}x_{j_2}\cdots x_{j_d}$, then $u_k=x_k\cdot x_{j_1}\cdots x_{j_p}$ and $u_r=x_r\cdot x_{j_1}\cdots x_{j_q}$ are both ${\bf t}$-spread monomials of $I$ with $p<q<d$. Then $x_r(u_k/x_k)$ is a ${\bf t}$-spread monomial of $I$ that divides $x_ru$ and $x_r(u_k/x_k)>_{\text{plex}}u_r$, an absurd. Hence $\max(u_r)\le\max(u_k)$, so $u_r$ divides $x_ru_k$ and again $g(x_ru_k)=u_r$.
	
	Thus, $g(x_ru_k)=u_r$ and
	\[\arraycolsep=1.4pt\def\arraystretch{2.2}
	\begin{array}{lrlrr}
	&\displaystyle\frac{x_ru_kv_k}{g(x_ru_k)x_{\max(u)}}&f(g(x_ru_k);\vartheta\setminus\{k,r\})&\phantom{aaaaa}&(\textup{as}\ u_kv_k=x_ku),\\
	=&\displaystyle\frac{x_rx_ku}{u_rx_{\max(u)}}&f(u_r;\vartheta\setminus\{k,r\})&\phantom{aaaaa}&(\textup{as}\ x_ru=u_rv_r),\\
	=&\displaystyle x_k\frac{v_r}{x_{\max(u)}}&f(u_r;\vartheta\setminus\{k,r\}).
	\end{array}
	\]
	This shows that $B_2=B_4$ and completes our proof.
\end{proof}

We consider the ideal in Example \ref{Ex:KoszVect2} and construct the differentials of its minimal free resolution. Note that in this case ${\bf t}=(1,0)$.
\begin{Expl}\label{Ex:KoszVect3}
	\rm Let $I\subset S=K[x_1,\dots,x_6]$ be the $(1,0)$-spread strongly stable ideal of Example \ref{Ex:KoszVect2} with minimal generating set $$G(I)=\big\{w_1=x_{1}x_{2},\,\,w_2=x_{1}x_{3},\,\,w_3=x_{1}x_{4}^2\big\}.$$
	By Example \ref{Ex:KoszVect2}, $\pd(S/I)=3$. Let
	$$
	\FF:0\rightarrow F_3\xrightarrow{\ d_3\ }F_2\xrightarrow{\ d_2\ }F_1\xrightarrow{\ d_1\ }F_0=S\xrightarrow{\ d_0\ }S/I\rightarrow0
	$$
	the minimal free resolution of $S/I$. We know that $d_0=\varepsilon:S\rightarrow S/I$ is the canonical map. We shall describe the differentials $d_1,d_2,d_3$ by appropriate monomial matrices.
	
	For $i=1,2,3$, the basis of the free $S$-modules $F_i=S\otimes_KH_i({\bf x})$ consists of
	$$
	f(w_j;\sigma)=(-1)^{(i-1)(i-2)/2}1\otimes[z(w_j;\sigma)],
	$$
	for $j=1,\dots,4$, $\sigma\subseteq[\max(w_j)-1]\setminus\supp_{\bf t}(w_j)$ and $|\sigma|=i-1$.\smallskip
	
	We introduce a natural order on the basis elements of $F_i$, as follows,
	$$
	f(w_i;\sigma)\succ f(w_j;\vartheta)\ \ \Longleftrightarrow\ \ i<j\ \ \textup{or}\ \ i=j\ \ \textup{and}\ \ e_{\sigma}>e_{\vartheta},
	$$
	where $e_{\sigma}>e_{\vartheta}$ with respect to the order on the wedge products defined in Section \ref{sec1}.
	
	For instance,
	\begin{equation}\label{eq:orderedBasisFiVectSpread}
	f(w_2;\{2\})\succ f(w_3;\{2\})\succ f(w_3;\{3\}).
	\end{equation}
	
	Then, $d_i$, $i=1,2,3$, may be represented by a matrix whose $j$th column is given by the components of $d_i(f_j)$ with respect to the ordered basis of $F_{i-1}$, where $f_j$ is the $j$th basis element of $F_i$ with respect to the order introduced.
	
	By equation (\ref{eq:diffVectSpreadRisMin}) we have that
	$$
	\begin{aligned}
	\FF:\,\,\,\,0\,\xrightarrow{}\,&\underset{}{F_3}\,\xrightarrow{
		\begin{pmatrix}
		-x_{4}^2\\
		x_{3}\\
		-x_{2}\\
		\end{pmatrix}}\,
	\underset{}{F_2}\,\xrightarrow{
		\begin{pmatrix}
		x_{3}^2&x_4^2&0\\
		-x_{2}&0&x_4^2\\
		0&-x_2&-x_3\\
		\end{pmatrix}}\\
	&\underset{}{F_1}\,\xrightarrow{
	\begin{pmatrix}
	x_1x_2&x_1x_3&x_2x_4^2\\
	\end{pmatrix}}
	\underset{}{F_0}
	\xrightarrow{
		\ d_0\
	}\,
	\underset{}{S/I}\,\rightarrow\,0.\\
	\end{aligned}
	$$
	
	For instance, taking into account the order given in (\ref{eq:orderedBasisFiVectSpread}), we have
	\begin{align*}
	d_3(f(w_3;\{2,3\}))\ =&\ \begin{pmatrix}
	-x_{4}^2\\
	x_{3}\\
	-x_{2}\\
	\end{pmatrix}\begin{pmatrix}
	1
	\end{pmatrix}=-x_4^2f(w_2;\{2\})+x_3f(w_3;\{2\})-x_2f(w_3;\{3\}),\\[3pt]
	d_2d_3(f(w_3;\{2,3\}))\ =&\ -x_4^2\big(-x_2f(w_2;\emptyset)+x_3f(w_1;\emptyset)\big)\\
	&\ +x_3\big(-x_2f(w_3;\emptyset)+x_4^2f(w_1;\emptyset)\big)\\
	&\ -x_2\big(-x_3f(w_3;\emptyset)+x_4^2f(w_2;\emptyset)\big)\ =\ 0.
	\end{align*}
\end{Expl}

\section{Generalized Algebraic Shifting theory}\label{sec5}

In this final Section, we extend algebraic shifting theory to vector-spread strongly stable ideals. From now on, $K$ is a field of characteristic zero. We recall that by the symbol $\Gin(I)$ we mean the \textit{generic initial ideal} of a monomial ideal $I\subset S$, with respect to the reverse lexicographic order, with $x_1>x_2>\cdots>x_n$ \cite{JT}. It is known that $\Gin(I)$ is a (${\bf 0}$-spread) strongly stable ideal.\smallskip

Firstly, we need some notions.

Let ${\bf t},{\bf s}\in\ZZ_{\ge0}^{d-1}$, ${\bf t}=(t_1,\dots,t_{d-1})$, ${\bf s}=(s_1,\dots,s_{d-1})$, with $d\ge 2$. We can transform any ${\bf t}$-spread monomial ideal into a ${\bf s}$-spread monomial ideal as follows: Let ${\bf 0}\in\ZZ_{\ge0}^{d-1}$ be the null vector with $d-1$ components. To denote the composition of functions
$$
\Mon(T;{\bf t})\xrightarrow{\sigma_{{\bf t},{\bf 0}}}\Mon(T;{\bf 0})\xrightarrow{\sigma_{{\bf 0},{\bf s}}}\Mon(T;{\bf s})
$$
we use the symbol $\sigma_{{\bf t},{\bf s}}$, where $T=K[x_1,x_2,\dots,x_n,\dots]$. Note that 
$\sigma_{{\bf t},{\bf s}}(1)=1$, $\sigma_{{\bf t},{\bf s}}(x_i)=x_i$, and for all monomials $u=x_{j_1}x_{j_2}\cdots x_{j_{\ell}}\in\Mon(T;{\bf t})$, $2\le\ell\le d$,
\begin{align*}
\sigma_{{\bf t},{\bf s}}(x_{j_1}x_{j_2}\cdots x_{j_{\ell}})=\textstyle\prod\limits_{k=1}^{\ell}x_{j_k-\sum_{r=1}^{k-1}t_r+\sum_{r=1}^{k-1}s_r}.
\end{align*}

Finally, for $I$ a ${\bf t}$-spread monomial ideal, we let $I^{\sigma_{{\bf t},{\bf s}}}$ the monomial ideal whose minimal generating set is $G(I^{\sigma_{{\bf t},{\bf s}}}):=\big\{\sigma_{{\bf t},{\bf s}}(u):u\in G(I)\big\}$. Note that $I^{\sigma_{{\bf t},{\bf s}}}=(I^{\sigma_{{\bf t},{\bf 0}}})^{\sigma_{{\bf 0},{\bf s}}}$.\medskip

As mentioned in the Introduction, we define the \textit{${\bf t}$-spread algebraic shifting} as follows: for $I$ a monomial ideal of $T$, we let $I^{s,{\bf t}}$ the following monomial ideal
$$
I^{s,{\bf t}}:=(\Gin(I))^{\sigma_{{\bf 0},{\bf t}}}.
$$
Note that for ${\bf t}={\bf 1}=(1,1,\dots,1)$, we obtain the classical algebraic shifting. Indeed, for ${\bf t}={\bf 1}$, $\sigma_{{\bf 0},{\bf t}}$ is the \textit{squarefree operator} defined in the Introduction.\smallskip

We are going to verify the following four properties:
\begin{enumerate}
	\item[]
	\begin{enumerate}
	\item[$(\textup{Shift}_1)$] $I^{s,{\bf t}}$ is a ${\bf t}$-spread strongly stable monomial ideal;
	\item[$(\textup{Shift}_2)$] $I^{s,{\bf t}}=I$ if $I$ is a ${\bf t}$-spread strongly stable ideal;
	\item[$(\textup{Shift}_3)$] $I$ and $I^{s,{\bf t}}$ have the same Hilbert function;
	\item[$(\textup{Shift}_4)$] If $I\subseteq J$, then $I^{s,{\bf t}}\subseteq J^{s,{\bf t}}$.
	\end{enumerate}
\end{enumerate}

\begin{Prop}\label{VectTiffVect0}
	Let $I$ be a monomial ideal. Then, $I$ is a ${\bf t}$-spread strongly stable ideal if and only if $I^{\sigma_{{\bf t},{\bf s}}}$ is a ${\bf s}$-spread strongly stable ideal.
\end{Prop}
\begin{proof}
	Suppose that $I$ is a ${\bf t}$-spread strongly stable ideal. Set $I'=I^{\sigma_{{\bf t},{\bf s}}}$. To show that $I'$ is a ${\bf s}$-spread strongly stable ideal, it suffices to check condition (ii) of Corollary \ref{cor:G(I)VectSpreadSS}. So, let $u\in G(I)$, $u=x_{j_1}x_{j_2}\cdots x_{j_d}$, then
	$$
	u_1=\sigma_{{\bf t},{\bf s}}(u)=\prod_{k=1}^{d}x_{j_k-\sum_{r=1}^{k-1}t_r+\sum_{r=1}^{k-1}s_r}=x_{j_1'}x_{j_2'}\cdots x_{j_d'}\in G(I')=G(I^{\sigma_{{\bf t},{\bf s}}}).
	$$
	Let $i\in\supp(u_1)$, $j<i$ such that $v_1=x_j(u_1/x_i)$ is ${\bf s}$-spread, we prove that $v_1\in I'$.
	
	Now, $i=j_\ell'=j_\ell-\sum_{r=1}^{\ell-1}t_r+\sum_{r=1}^{\ell-1}s_r$, for some $\ell\in\{1,\dots,d\}$, and $j_{p-1}'+s_{p-1}\le j\le j_p'-1$, for some $p\le\ell$, in particular for $p=1$, $j<j_1'$. Hence,
	$$
	v_1=x_j(u_1/x_i)=\big(\prod_{k=1}^{p-1}x_{j_k'}\big)x_j\big(\prod_{k=p}^{\ell-1}x_{j_k'}\big)\big(\prod_{k=\ell+1}^d x_{j_k'}\big).
	$$
	Recall that $\sigma_{{\bf s},{\bf t}}$ is the inverse map of $\sigma_{{\bf t},{\bf s}}$. Set $v=\sigma_{{\bf s},{\bf t}}(v_1)$, then $v$ is ${\bf t}$-spread, and
	$$
	v=\sigma_{{\bf s},{\bf t}}(v_1)=\big(\prod_{k=1}^{p-1}x_{j_k}\big)x_{j-\sum_{r=1}^{p-1}s_r+\sum_{r=1}^{p-1}t_r}\big(\prod_{k=p}^{\ell-1}x_{j_k-s_k+t_k}\big)\big(\prod_{k=\ell+1}^d x_{j_k}\big).
	$$
	Since $j_{k+1}-j_k\ge t_k$ for all $k$ and $j_p'=j_p-\textstyle\sum_{r=1}^{p-1}t_r+\textstyle\sum_{r=1}^{p-1}s_r$, we have
	\begin{align}
	\label{i<jVect1} j_k-s_k+t_k\le j_{k+1}-s_k\ &\le\ j_{k+1},\ \ \text{for all}\ k=p,\dots,\ell-1,\ \text{and}\\
	\label{i<jVect2} j-\textstyle\sum_{r=1}^{p-1}s_r+\textstyle\sum_{r=1}^{p-1}t_r\ &<\ j_p'-\textstyle\sum_{r=1}^{p-1}s_r+\textstyle\sum_{r=1}^{p-1}t_r=j_p.
	\end{align}
	Setting 
	\begin{align*}
	z_m\ &=\ \begin{cases}
	x_{j-\sum_{r=1}^{p-1}s_r+\sum_{r=1}^{p-1}t_r}(u/x_{j_p}),&\text{for}\ m=1,\\
	x_{j_{(p+m-2)}-s_{(p+m-2)}+t_{(p+m-2)}}(z_{m-1}/x_{j_{(p+m-1)}}),&\text{for}\ m=2,\dots,\ell+1-p,
	\end{cases}
	\end{align*}
	we see that the monomials $z_m$ are ${\bf t}$-spread. Moreover, as $I$ is ${\bf t}$-spread strongly stable, $z_1\in I$ by (\ref{i<jVect2}), and inductively $z_m\in I$, by (\ref{i<jVect1}). So, $v=z_{\ell+1-p}\in I$, and by Lemma \ref{LemG(I)M(I)}, $v=w_1w_2$ for unique monomials $w_1\in G(I)$, $w_2$ such that $\max(w_1)\le\min(w_2)$. Hence, $\sigma_{{\bf t},{\bf s}}(w_1)$  divides $v_1=\sigma_{{\bf t},{\bf s}}(v)$, with $\sigma_{{\bf t},{\bf s}}(w_1)\in G(I')=G(I^{\sigma_{{\bf t},{\bf s}}})$. Finally, $v_1\in I'=I^{\sigma_{{\bf t},{\bf s}}}$, as desired. The converse is trivially true as $I,{\bf t},{\bf s}$ are arbitrary.
\end{proof}

By virtue of this proposition, the property $(\textup{Shift}_1)$ is verified. Indeed, it is known that $\Gin(I)$ is a ${\bf 0}$-spread strongly stable ideal \cite{JT}. Consequently, $I^{s,{\bf t}}$ is a ${\bf t}$-spread strongly stable ideal, as desired.

The operators $\sigma_{{\bf t},{\bf s}}$ behave well, in fact they preserve the graded Betti numbers. We first note that Theorem \ref{TeorTSpreadStronlgyStableBetti} implies a formula for the graded Betti numbers. We remark that the next result holds whatever the characteristic of the field $K$ is.

\begin{Cor}\label{Cor:BettiNumbFormulaVectSpread}
	Let $I$ be a ${\bf t}$-spread strongly stable ideal of $S$. Then,
	\begin{equation}\label{eq:FormulaBettiVectSpread}
	\beta_{i,i+j}(I)=\sum_{u\in G(I)_j}\binom{\max(u)-1-\sum_{\ell=1}^{j-1}t_\ell}{i},\ \ \ \text{for all}\ \ i,j\ge0.
	\end{equation}
	In particular, the graded Betti numbers of a vector-spread strongly stable ideal $I\subset S$ do not depend upon the characteristic of the field $K$.
\end{Cor}
\begin{proof}
	Let $i,j\ge0$. By equation (\ref{eq:betaIS/I}), $\beta_{i,i+j}(I)=\beta_{i+1,i+j}(S/I)=\dim_KH_{i+1}({\bf x};S/I)_{i+j}$. By Theorem \ref{TeorTSpreadStronlgyStableBetti}, the degree of a basis element $\big[e(u;\sigma)\big]$ of $H_{i+1}({\bf x};S/I)_{i+j}$ is given by $|\sigma|+1+\deg(u)-1=i+j$. Thus $u\in G(I)_j$. For a fixed $u\in G(I)_j$, we have $\sigma\subseteq[\max(u)-1]\setminus\supp_{\bf t}(u)$. Hence, there are $\binom{|[\max(u)-1]\setminus\supp_{\bf t}(u)|}{i}=\binom{\max(u)-1-\sum_{\ell=1}^{j-1}t_\ell}{i}$ possible choices for $\sigma$. Summing all these binomials over $u\in G(I)_j$, we obtain the formula in the statement.
\end{proof}

Suitable choices of ${\bf t}$ return several well known formulas for the graded Betti numbers. ${\bf t}=(0,0,\dots,0)$ returns the Eliahou--Kervaire formula for (strongly) stable ideals \cite{EK}; ${\bf t}=(1,1,\dots,1)$ gives the Aramova--Herzog--Hibi formula for squarefree (strongly) stable ideals \cite{AHH2}. Finally, in the \textit{uniform} case, \emph{i.e.}, ${\bf t}=(t,t,\dots,t)$, we have the Ene--Herzog--Qureshi formula for uniform $t$-spread strongly stable ideals \cite{EHQ}.\medskip

Let $P_I^S(y)=\sum_i\beta_i(I)y^i$ be the Poincar\'e series of $I$. Equation (\ref{eq:FormulaBettiVectSpread}) implies
\begin{Cor}
	Let $I\subset S$ be a ${\bf t}$-spread strongly stable ideal. Then
	\begin{enumerate}
		\item[\textup{(a)}] $P_I^S(y)=\sum_{u\in G(I)}(1+y)^{\max(u)-1-\sum_{\ell=1}^{\deg(u)-1}t_\ell}$;
		\item[\textup{(b)}] $\pd(I)=\max\big\{\max(u)-1-\textstyle\sum_{j=1}^{\deg(u)-1}t_j:u\in G(I)\big\}$;
		\item[\textup{(c)}] $\reg(I)=\max\big\{\deg(u):u\in G(I)\big\}$.
	\end{enumerate}
\end{Cor}

Let us return now to our shifting operators. As announced, we have
\begin{Lem}\label{lemmaSigmaConserva}
	Let $I$ be a ${\bf t}$-spread strongly stable ideal. Then $I^{\sigma_{{\bf t},{\bf s}}}$ is a ${\bf s}$-spread strongly stable ideal, and for all $i,j\ge0$,
	$$
	\beta_{i,i+j}(I)=\beta_{i,i+j}(I^{\sigma_{{\bf t},{\bf s}}}).
	$$
\end{Lem}
\begin{proof}
	We have just proved that $I^{\sigma_{{\bf t},{\bf s}}}$ is a ${\bf s}$-spread strongly stable ideal with minimal generating set $G(I^{\sigma_{{\bf t},{\bf s}}})=\{{\sigma_{{\bf t},{\bf s}}}(u):u\in G(I)\}$. Moreover, for $u\in G(I)$, we have $\max(\sigma_{{\bf t},{\bf s}}(u))=\max(u)-\sum_{\ell=1}^{\deg(u)-1}t_\ell+\sum_{\ell=1}^{\deg(u)-1}s_\ell$. Hence, Corollary \ref{Cor:BettiNumbFormulaVectSpread} yields
	\begin{align*}
	\beta_{i,i+j}(I^{\sigma_{{\bf t},{\bf s}}})\ &=\ \sum_{\sigma_{{\bf t},{\bf s}}(u)\in G(I^{\sigma_{{\bf t},{\bf s}}})_j}\binom{\max({\sigma_{{\bf t},{\bf s}}}(u))-1-\sum_{\ell=1}^{j-1}s_\ell}{i}\\
	&=\ \sum_{u\in G(I)_j}\binom{\max(u)-\sum_{\ell=1}^{j-1}t_\ell+\sum_{\ell=1}^{j-1}s_\ell-1-\sum_{\ell=1}^{j-1}s_\ell}{i}\\
	&=\ \sum_{u\in G(I)_j}\binom{\max(u)-1-\sum_{\ell=1}^{j-1}t_\ell}{i}\ =\ \beta_{i,i+j}(I).
	\end{align*}
\end{proof}

As a consequence, the property $(\textup{Shift}_3)$ is verified too. Indeed, it is known that $I$ and $\Gin(I)$ have the same Hilbert function. Moreover, by Lemma \ref{lemmaSigmaConserva}, $\Gin(I)$ and $(\Gin(I))^{\sigma_{{\bf 0},{\bf t}}}$ have the same graded Betti numbers and thus the same Hilbert function.\smallskip

Note that condition $(\textup{Shift}_4)$ is trivially verified. Finally it remains to establish condition $(\textup{Shift}_2)$. This is accomplished in the next theorem.

\begin{Thm}\label{GinVectSpread}
	Let $K$ be a field of characteristic zero. Let $I\subset S$ be a ${\bf t}$-spread strongly stable ideal. Then
	$$
	I=(\Gin(I))^{\sigma_{{\bf 0},{\bf t}}}.
	$$
\end{Thm}
\begin{proof}
	We proceed by induction on the integer $\ell=\max\{\max(u):u\in G(I)\}\ge1$. If $\ell=1$, then $G(I)=\{x_1^a\}$, $I=(x_1^a)$, and $\Gin(I)=I=(x_1^a)$, moreover $\sigma_{{\bf t},{\bf 0}}(x_1^a)=x_1^a$, for some $a\ge1$. So, the thesis holds for $\ell=1$.
	
	Let $\ell>1$. By \cite[Lemma 11.2.8]{JT} we can assume $\ell=n$. So, there exists a monomial $u\in G(I)$ with $\max(u)=n$. Let $p=\max\{p:x_n^p\ \textup{divides}\ w\ \textup{for some}\ w\in G(I)\}$, our hypothesis implies that $p\ge1$. We consider the following ideals:
	\begin{align*}
	I'&=I:(x_n^p),\ \ \ \ I''=\big(u\in G(I):\max(u)<n\big).
	\end{align*} 
	Both are again ${\bf t}$-spread strongly stable ideals, and $I''\subseteq I\subseteq I'$. By inductive hypothesis, $\Gin(I')=(I')^{\sigma_{{\bf t},{\bf 0}}}$ and $\Gin(I'')=(I'')^{\sigma_{{\bf t},{\bf 0}}}$. Equivalently, $$I'=\Gin(I')^{\sigma_{{\bf 0},{\bf t}}}\ \ \textup{and}\ \ I''=\Gin(I'')^{\sigma_{{\bf 0},{\bf t}}}.$$ Therefore, $I''\subseteq\Gin(I)^{\sigma_{{\bf 0},{\bf t}}}\subseteq I'$.
	
	\textsc{Claim 2.} It is
	\begin{equation}\label{eq:claimgins}
	I\subseteq\Gin(I)^{\sigma_{{\bf 0},{\bf t}}}.
	\end{equation}
	
	To prove \textsc{Claim 2}, it is enough to show that each $u\in G(I)$ with $\max(u)=n$ belongs to $\Gin(I)^{\sigma_{{\bf 0},{\bf t}}}$. Indeed, since $I''\subseteq\Gin(I)^{\sigma_{{\bf 0},{\bf t}}}$, all monomials $u\in G(I)$ with $\max(u)<n$ are in $\Gin(I)^{\sigma_{{\bf 0},{\bf t}}}$.
	
	Let $u\in G(I)$ with $\max(u)=n$. We set $a=n-1-\sum_{j=1}^{\deg(u)-1}t_j$ and $b=a+\deg(u)$. By Corollary \ref{Cor:BettiNumbFormulaVectSpread} we have
	\begin{align*}
	\beta_{a,b}(I)\ &=\ \sum_{\substack{v\in G(I)\\ \deg(v)=\deg(u)}}\binom{\max(v)-1-\sum_{j=1}^{\deg(u)-1}t_j}{n-1-\sum_{j=1}^{\deg(u)-1}t_j}\\
	&=\ \big|\big\{v\in G(I):\max(v)=n,\ \deg(v)=\deg(u) \big\}\big|.
	\intertext{Similarly, as $\Gin(I)^{\sigma_{{\bf 0},{\bf t}}}$ is ${\bf t}$-spread strongly stable,}
	\beta_{a,b}(\Gin(I)^{\sigma_{{\bf 0},{\bf t}}})\ &=\ \sum_{\substack{w\in G(\Gin(I)^{\sigma_{{\bf 0},{\bf t}}})\\ \deg(w)=\deg(u)}}\binom{\max(w)-1-\sum_{j=1}^{\deg(u)-1}t_j}{n-1-\sum_{j=1}^{\deg(u)-1}t_j}\\
	&=\ \big|\big\{w\in G(\Gin(I)^{\sigma_{{\bf 0},{\bf t}}}):\max(w)=n,\ \deg(w)=\deg(u) \big\}\big|.
	\end{align*}
	Moreover, by \cite[Corollary 3.3.3]{JT} and by Lemma \ref{lemmaSigmaConserva}, we have
	$$
	\beta_{a,b}(I)\ \le\ \beta_{a,b}(\Gin(I))\ =\ \beta_{a,b}(\Gin(I)^{\sigma_{{\bf 0},{\bf t}}}).
	$$
	Hence
	\begin{equation}\label{setsdisug}
	\begin{aligned}
	\big|\big\{w\in G(\Gin(I)^{\sigma_{{\bf 0},{\bf t}}}):\max(w)=n,\deg(w)=\deg(u) \big\}\big|\ge\\ \ \ \ \ \ \ \big|\big\{v\in G(I):\max(v)=n,\deg(v)=\deg(u) \big\}\big|.
	\end{aligned}
	\end{equation}
	Our aim is to prove that $u\in G(I)$ with $\max(u)=n$ belongs to $\Gin(I)^{\sigma_{{\bf 0},{\bf t}}}$.\smallskip
	
	Let $w_1,\dots,w_s$ be the monomial generators in $G(\Gin(I)^{\sigma_{{\bf 0},{\bf t}}})$ such that $\max(w_i)=n$ and $\deg(w_1)\le\deg(w_2)\le\dots\le\deg(w_s)$. Since $\Gin(I)^{\sigma_{{\bf 0},{\bf t}}}\subseteq I'$, we have $w_i x_n^p\in I$, for all $i=1,\dots,s$. We prove that $w_i\in I$ for all $i$. Since $w_ix_n^p\in I$, there is a monomial $v_i\in G(I)$ such that $v_i$ divides $w_i x_n^p$. We have $\deg(v_i)\le\deg(w_i)+p$, for all $i=1,\dots,s$.\smallskip
	
	If $\deg(v_1)<\deg(w_1)$, setting $u=v_1$ in (\ref{setsdisug}), we would have an absurd. Hence $\deg(v_1)\ge\deg(w_1)$. By finite induction, $\deg(v_i)\ge\deg(w_i)$, for all $i=1,\dots,s$.\smallskip
	
	Now, if $\deg(v_s)\ge\deg(w_s)+1$, setting $u=v_s$ in (\ref{setsdisug}), we would obtain an absurd. Hence, $\deg(v_s)\le\deg(w_s)$, and since we have proved that $\deg(v_s)\ge\deg(w_s)$, we obtain $\deg(v_s)=\deg(w_s)$. Iterating this argument, $\deg(v_i)=\deg(w_i)$, for all $i=1,\dots,s$.\smallskip
	
	If $v_i=(w_i x_n^p)/x_n^p=w_i$ we set $u_i=v_i$ and note that $u_i=w_i$ divides $w_i$. Otherwise, $v_i=(w_i x_n^p)/z_i$ for some monomial $z_i\ne x_n^p$, we note that $v_i$ has bigger sorted indexes than $w_i$, thus since $I$ is ${\bf t}$-spread strongly stable $w_i\in I$. Hence, there is a monomial $u_i\in G(I)$ that divides $w_i$. Finally, we have constructed monomials $u_1,\dots,u_s\in G(I)$ such that $u_i$ divides $w_i$, for all $i=1,\dots,s$. Repeating the same argument as before, using (\ref{setsdisug}), we see that $\deg(u_i)\ge\deg(w_i)$, for all $i$, hence $u_i=w_i$, since $u_i$ divides $w_i$, for all $i=1,\dots,s$.\smallskip
	
	Thus, $w_i=u_i\in G(I)$, for all $i=1,\dots,s$, and we get the inclusion
	$$
	\big\{w\in G(\Gin(I)^{\sigma_{{\bf 0},{\bf t}}}):\max(w)=n\big\}\subseteq\big\{u\in G(I):\max(u)=n\big\}.
	$$
	This equation together with (\ref{setsdisug}) yield
	$$
	\big\{w\in G(\Gin(I)^{\sigma_{{\bf 0},{\bf t}}}):\max(w)=n\big\}=\big\{u\in G(I):\max(u)=n\big\}.
	$$
	Hence, \textsc{Claim 2} is true.\smallskip
	
	Finally, $I$ and $\Gin(I)$ have the same Hilbert function. Moreover, by Lemma \ref{lemmaSigmaConserva}, $\Gin(I)$ and $\Gin(I)^{\sigma_{{\bf 0},{\bf t}}}$ have the same Hilbert function. Hence $I$ and $\Gin(I)^{\sigma_{{\bf 0},{\bf t}}}$ have the same Hilbert function. Formula (\ref{eq:claimgins}) and this observation imply that $I=\Gin(I)^{\sigma_{{\bf 0},{\bf t}}}$, or equivalently $\Gin(I)=I^{\sigma_{{\bf t},{\bf 0}}}$, proving the theorem.
\end{proof}

We end the paper by remarking that the operator $\sigma_{{\bf t},{\bf s}}$ establishes a bijection between ${\bf t}$-spread strongly stable ideals and ${\bf s}$-spread strongly stable ideals.

\end{document}